\documentclass[a4paper]{article}

\usepackage{amsfonts, amsmath, amsthm, empheq, amssymb, color, indentfirst, physics}

\allowdisplaybreaks[4]
\newtheorem{thm}{Theorem}[section]

\newtheorem{lem}{Lemma}[section]

\renewenvironment{abstract}{%
        \small
        \quotation
         \noindent {\bfseries \abstractname } }%
      {\if@twocolumn\else\endquotation\fi}

\def\Im{\mathrm{\,Im\,}}

\title{\Large\bf
Well-posedness and asymptotic estimate for a diffusion equation with time-fractional derivative}

\author{\large Zhiyuan LI $^1$, Xinchi HUANG $^2$, Masahiro YAMAMOTO$^3$}

\date{}

\begin{document}
\maketitle

\renewcommand{\thefootnote}{\fnsymbol{footnote}}
\footnotetext{\hspace*{-5mm} 
\begin{tabular}{@{}r@{}p{16cm}@{}} 
$^1$ 
& School of Mathematics and Statistics,
Shandong University of Technology,
Zibo, Shandong 255049, China. 
E-mail: zyli@sdut.edu.cn\\
$^2$
& Graduate School of Mathematical Sciences, 
the University of Tokyo,
3-8-1 Komaba, Meguro-ku, Tokyo 153-8914, Japan.\\
& JSPS Postdoctoral Fellowships for research in Japan.
E-mail: huangxc@ms.u-tokyo.ac.jp\\
$^3$
& Graduate School of Mathematical Sciences, 
the University of Tokyo,
3-8-1 Komaba, Meguro-ku, Tokyo 153-8914, Japan. 
E-mail: myama@ms.u-tokyo.ac.jp\\
& Honorary Member of Academy of Romanian Scientists,
Ilfov, nr. 3, Bucuresti, Romania.\\
& Correspondence member of Accademia Peloritana dei Pericolanti,
Palazzo Universit\'a, Piazza S. Pugliatti 1 98122 Messina, Italy.\\
& Peoples' Friendship University of Russia (RUDN University), 
6 Miklukho-Maklaya St, Moscow, 117198, Russia.
\end{tabular}}

\begin{abstract}
In this paper, we study the asymptotic estimate of 
solution for a mixed-order time-fractional diffusion equation in a bounded domain subject to 
the homogeneous Dirichlet boundary condition. 
Firstly, the unique existence and regularity estimates of solution to the initial-boundary value problem 
are considered. Then combined with some important properties, including a maximum principle for 
a time-fractional ordinary equation and a coercivity inequality for fractional derivatives, the energy 
method shows that the decay in time of the solution is dominated by the term $t^{-\alpha}$ as 
$t\to\infty$. 

\vskip 4.5mm

\noindent\begin{tabular}{@{}l@{ }p{11.5cm}} {\bf Keywords } &
mixed-order fractional diffusion equation,
initial-boundary value problem,   
asymptotic estimate,
energy method
\end{tabular}

\vskip 4.5mm


\end{abstract}

\section{Introduction}
\label{intro}

Within the last few decades, an abundance of anomalous processes was observed and 
confirmed by more and more experiments in several different application areas in natural sciences, 
e.g., biology, geological sciences, medicine, see \cite{Hilfer00}, \cite{Ko08}, \cite{Ko11}, \cite{SKB02}, 
\cite{U13} and the references therein. 
For example, to characterize these diffusion processes, an important micro statistic quantity--
the mean square displacement which describes how fast particles diffuse was used. 
In the most of anomalous diffusion cases, one observes a fractional power-law mean square 
displacement (cf. \cite{MK00}), which shows that the diffusion is slower than that in the classical 
diffusion case.

For the mathematical studies of the anomalous diffusion, we refer to \cite{RA94} in which 
some macro models in the form of fractional diffusion equations were derived by the technique of 
the continuous-time random walk under some suitable conditions posed on the probability density 
functions for the jumps length and the waiting times between two successive jumps. 
In most of the cases, these macro models have a form of the single or multi-term time-, space-, or 
time-space-fractional differential equations (see e.g., \cite{C95}, \cite{Luc11a}, \cite{Luc11b} and 
the references therein for the details). 

In this paper, we consider one of the important special cases: the time-fractional diffusion equation, 
which attracts great attention from many aspects during the last years, mostly due to their applications 
in the modeling of anomalous diffusion. 
For example, we mention important applications on some amorphous semiconductors 
\cite{MK00}, \cite{U13}, the modeling of dynamic processes in polymer materials, 
heat conduction with memory \cite{P93} and the diffusion in fluids in porous media 
\cite{C99}, \cite{HH98}, \cite{Ja01}. 
We refer to \cite{CLY17}, \cite{HLY19}, \cite{LHY20}, \cite{LLY15} and \cite{L09}, for 
the mathematical theory of the fractional differential equations, whereas we refer to 
\cite{GOS18}, \cite{JLZ16}, \cite{LX16}, \cite{PS16}, \cite{SLC12}, \cite{SZT17} for the numerical study. 

The main goal of this paper is to establish decay estimates for the solution of our mixed-order 
fractional diffusion equation by energy method. 
As is known, the asymptotic behavior of solutions to the equations which describes some physical 
processes is important both by itself and as a basis for analysis of the suitable numerical methods for 
the solutions and the inverse problems for these equations. 
Researches are rapidly growing on the asymptotic behavior for the time-fractional diffusion equations 
and we only give a brief and typical review of the existing works instead of a comprehensive list. 
The asymptotic behavior as $t\to\infty$ for the single or multi-term time-fractional diffusion equations 
in a bounded domain was studied in \cite{LHY20}, \cite{LLY15} and \cite{SY11}, 
where the decay of solutions is dominated by the lowest order of the fractional derivatives. 
The proof of this fact is based on an explicit representation formula for the solution by the 
Fourier expansion method. From this explicit formula, by evaluating the inversion transform of 
the solutions, the decay in time of the solutions can be obtained. 
In the unbounded domain, we refer to \cite{CLY17} and \cite{KSZ}, where properties of several special 
functions, for example, the H-functions, the Mittag-Leffler functions, were used to obtain the formula of 
the solution. 
It turns out that all the above arguments heavily rely on the explicit representation of the solution. 
Indeed, the coefficients of the equation are required to be $t$-independent at least, hence the Fourier 
method works and derivation of explicit representation formula of the solution becomes possible.

In this paper, we continue the research activities initiated in \cite{LHY20}, \cite{LLY15} and \cite{SY11}, 
and employ an energy method to deal with the fractional diffusion equation with $t$-dependent 
coefficients for which there are no explicit representation formula for the solution. 
We will see that the $L^2(\Omega)$-norm of the solution is dominated by $c_\alpha t^{-\alpha}$. 
This energy method has been widely used to deal with the asymptotic estimate for other types of 
evolution equations, see e.g., \cite{Kubica} and \cite{VZ} and the references therein.

The rest of the paper is organized as follows: 
In Section {\bf 2}, we formulate our problem and show our main results including the well-posedness and 
the long-time asymptotic behavior of the solution to the initial-boundary value problem for 
the time-fractional diffusion equation. 
The proof of the well-posedness of the solution is given in Section {\bf 3}, 
whereas the long-time asymptotic estimate is proved in Section {\bf 4}. 
Finally, the last section is devoted to the conclusions and some open problems. 

\section{Problem formulation and main results}
\label{sect-intro}
In this paper, let $T>0$ and $\Omega$ be an open bounded domain in $\mathbb{R}^d$ with 
a smooth boundary $\partial\Omega$. 
We deal with the time-fractional differential equation
\begin{equation}
\label{equ-u}
\partial_t u + q(t)\partial_t^\alpha u
= - Au + c(x,t)u + f(x,t), \ \ (x,t) \in \Omega\times(0,T)
\end{equation}
with the initial-boundary value
\begin{equation}
\label{equ-ibc}
\left\{ \begin{array}{l}
u(x,0)=u_0(x), \ \ x\in  \Omega, \\
u(x,t)=0,\ \ (x,t) \in \partial\Omega\times(0,T),
\end{array}\right.
\end{equation}
where the coefficients $q,c$ are smooth enough, e.g., 
$c\in L^\infty(0,T;W^{2,\infty}(\Omega))$, $q\in L^\infty(0,T)$ 
and $A$ is a symmetric uniformly elliptic operator defined by 
$$
A u(x) := -\sum_{i,j=1}^d \frac{\partial}{\partial x_i}  \left(a_{ij}(x)\frac{\partial}{\partial x_j} u(x)\right),
\quad u\in D(A):= H_0^1(\Omega)
$$
with $a_{ij}(x) = a_{ji}(x)$, $1\leq i,j \leq d,\ x\in \overline\Omega$
and $a_{ij}\in C^1(\overline{\Omega})$ such that
$$
\sum_{i,j=1}^d a_{ij}(x)\xi_i \xi_j
\geq \nu|\xi|^2 ,\quad\forall x\in \overline\Omega,\ \forall\xi=(\xi_1,\cdots,\xi_d )\in\Bbb{R}^d
$$
for some constant $\nu > 0$. 
By $\partial^{\alpha}_t$ we denote the Caputo fractional derivative of order $\alpha\in (0,1)$:
$$
\partial^{\alpha}_t \varphi(t) := 
\frac{1}{\Gamma(1-\alpha)} \int_0^t (t-\tau)^{-\alpha}\, \frac{d}{d\tau} \varphi(\tau)\, d\tau.
$$
Here and henceforth $L^2(\Omega)$, $H^1(0,T)$, $H^1(\Omega)$ and $H_0^1(\Omega)$ denote 
the usual Lebesgue space and the Sobolev spaces, 
and $H^{-1}(\Omega)$ denotes the dual space of $H_0^1(\Omega)$. 
Meanwhile we denote $\|\cdot\|_{L^2(\Omega)}$, $\|\cdot\|_{H^1(0,T)}$, $\|\cdot\|_{H^1(\Omega)}$, 
$\|\cdot\|_{H_0^1(\Omega)}$ and $\|\cdot\|_{H^{-1}(\Omega)}$ as the corresponding norms. 

In this paper, we mainly discuss the unique existence and the long-time asymptotic behavior of 
the solution to the initial-boundary value problem \eqref{equ-u}--\eqref{equ-ibc}. 

Our main results are presented in Theorems \ref{thm-uniq-exist} and \ref{thm-asymp-long} 
formulated below and the proofs are given in Sections {\bf 3} and {\bf 4}. 
We start with a result of the unique existence and the regularity of the solution. 
For arbitrarily fixed $T>0$, we have the following theorem. 
\begin{thm}
\label{thm-uniq-exist}
Let $u_0\in L^2(\Omega)$, $f\in L^2(0,T;H^{-1}(\Omega))$. Then there exists a unique solution 
$u\in H^1(0,T; H^{-1}(\Omega)) \cap L^2(0,T; H_0^1(\Omega)) \cap C([0,T];L^2(\Omega))$ 
to the initial-boundary value problem \eqref{equ-u}--\eqref{equ-ibc}, 
and there exists a constant $C_1>0$ such that
\begin{equation}
\label{esti-fst}
\left\|u\right\|_{H^1(0,T;H^{-1}(\Omega))} + \|u\|_{L^2(0,T;H_0^1(\Omega))}
\le C_1\left(\|u_0\|_{L^2(\Omega)} + \|f\|_{L^2(0,T;H^{-1}(\Omega))}\right).
\end{equation}
In addition, we assume that $u_0\in H_0^1(\Omega)$ and $f\in L^2(0,T; L^2(\Omega))$. 
Then the solution $u$ further belongs to 
$H^1(0,T; L^2(\Omega)) \cap L^2(0,T; H^2(\Omega)\cap H_0^1(\Omega))$, 
and there exists a constant $C_2>0$ satisfying
\begin{equation}
\label{esti-sec}
\left\|u\right\|_{H^1(0,T;L^2(\Omega))} + \|u\|_{L^2(0,T;H^2(\Omega))}
\le C_2\left(\|u_0\|_{H_0^1(\Omega)} + \|f\|_{L^2(0,T;L^2(\Omega))}\right).
\end{equation}
Furthermore, we assume that $u_0\in H^2(\Omega)\cap H_0^1(\Omega)$ and 
$f\in H^1(0,T; L^2(\Omega))$. 
Then $u\in W^{1,\infty}(0,T;L^2(\Omega))\cap L^\infty(0,T;H^2(\Omega)\cap H_0^1(\Omega))$, 
and there exists a constant $C_3>0$ satisfying
\begin{align*}
\mathrm{ess} \hspace{-0.2cm}\sup_{0\le t\le T} \left(\|\partial_t u(t)\|_{L^2(\Omega)} 
+ \|u(t)\|_{H^2(\Omega)} \right) 
\le C_3\left( \|u_0\|_{H^2(\Omega)} + \|f\|_{H^1(0,T;L^2(\Omega))} \right).
\end{align*}
\end{thm}
Here the constants $C_1,C_2,C_3$ depend on $\alpha, T, \nu$ and some norms of 
the coefficients $c$, $a_{ij}$ and $q$. 
Moreover, we mention that by a classical result for parabolic equations, we have actually 
$u\in C^1((0,T];L^2(\Omega))$ provided that the coefficients $q, c$ are sufficiently smooth, 
but we do not discuss the details here.

Next we propose the result for the long-time asymptotic behavior. 
For the asymptotic estimate, we assume further that $q,c$ are continuous in time variable, 
$q_0 \leq q(t) \leq q_1,\ t>0$ for some positive constants $q_1\ge q_0>0$ and 
$c(x,t)\leq 0$, $(x,t)\in \Omega\times (0,\infty)$. 
Then we have
\begin{thm}
\label{thm-asymp-long}
Assume that $f=0$, $u_0\in L^2(\Omega)$ and $q,c$ satisfy the above conditions. 
Let $u$ be the solution to the initial-boundary value problem \eqref{equ-u}--\eqref{equ-ibc}. 
Then for arbitrarily fixed $t_0>0$, there exists a constant $C>0$, depending only on 
$\alpha, q_1, \nu,\Omega$ and $t_0$, such that the following long-time asymptotic estimate
$$
\|u(\,\cdot\,,t)\|_{L^2(\Omega)}
\leq C\|u_0\|_{L^2(\Omega)}t^{-\alpha} 
$$
holds true for any $t\ge t_0$. 
\end{thm}
We also mention that the decay rate is the best possible. 
In fact, we can consider a special case that $q$ is a positive constant, $c$ is a nonnegative constant 
and $a_{ij}=\delta_{ij}$, then by the eigenfunction expansion method, 
we find the long-time asymptotic behavior of the solution 
is exactly $t^{-\alpha}$.

For a diffusion equation with time-fractional derivatives, in general, 
the decay rate is characterized by the lowest order of the derivatives (see e.g., \cite{LLY15}). 
In this paper, for simplicity, we are devoted to the case of two time derivatives with the orders 
$1$ and $\alpha\in (0,1)$, and 
the decay rate is never exponential unlike the case of only the first-order time derivative, 
but is $t^{-\alpha}$ as Theorem 2.2 proves. 
Moreover, if we consider the equation \eqref{equ-u} with more than one time-fractional derivatives, 
then the decay rate is subject to the lowest order, which we describe as a concluding remark in 
Section {\bf 5}.

\section{Unique existence and regularity of solution} \label{sec-unique-existence}
In this section, we first prove the unique existence of solution to the initial-boundary value problem 
{\eqref{equ-u}--\eqref{equ-ibc} in the space $H^{\alpha_1}(0,T; H^{-1}(\Omega))$ with arbitrarily fixed 
$1>\alpha_1>\max\{\alpha,\frac12\}$.
The proof is based on the classical unique existence of solution to parabolic equations and 
the Fredholm alternative. 
Next we also propose some improved regularity of the solution and establish the related estimates 
by employing generalized Gr\"onwall's inequality.

\subsection{Preliminary}
Before giving the proofs of our main results, we start with some useful representations of the solution 
$u$ to the initial-boundary value problem \eqref{equ-u}--\eqref{equ-ibc}. 

Because of the conditions imposed on the elliptic operator ${A}$, there exists a
system of eigenfunctions: 
$\{\varphi_k\}_{k=1}^\infty$, $\varphi_k \in H^2(\Omega)\cap H^1_0(\Omega)$ 
which satisfies the relations 
${A}\varphi_k=\lambda_k\varphi_k,\ k=1,2,\ldots$ 
and forms an orthonormal basis of $L^2(\Omega)$. 
The corresponding eigenvalues 
$\lambda_k,\ k=1,2,\ldots$ are all positive: 
$0<\lambda_1\le \lambda_2\leq\ldots$ and $\lambda_k\to\infty$ as $k\to\infty$. 
Henceforth, $\langle \cdot ,\cdot \rangle$ denotes the scalar product between 
$H^{-1}(\Omega)$ and $H_0^1(\Omega)$.

Moreover, we define the Mittag-Leffler function by
$$
E_{\alpha,\gamma}(z):=\sum_{k=0}^\infty \frac{z^k}{\Gamma(\alpha k+\gamma)},\ z\in\mathbb C,
$$
where $\alpha,\gamma>0$ are arbitrary constants. 
The following useful lemmata hold.
\begin{lem}
Let the constants $\alpha\in (0,1)$ and $\mu>0$ be given. Then the following equalities
\begin{equation}
\label{eq-ml1}
\partial_t E_{\alpha,1}(-\mu t^\alpha) = -\mu t^{\alpha-1} E_{\alpha,\alpha}(-\mu t^\alpha),
\end{equation}
and
\begin{equation}
\label{eq-ml2}
\partial_t^\alpha E_{\alpha,1}(-\mu t^\alpha) = -\mu E_{\alpha,1}(-\mu t^\alpha),
\end{equation}
are valid for any $t>0$.
\end{lem}
We refer to Podlubny \cite{P99} for the proof.

\begin{lem}
\label{lem-ML}
Let $0<\alpha<2$ and $\gamma>0$. We suppose that $\frac{\pi\alpha}{2}<\mu<\min\{\pi,\pi\alpha\}$. 
Then there exists a constant $C= C(\alpha,\gamma,\mu)>0$ such that
\begin{equation}
\label{esti-ML}
|E_{\alpha,\gamma}(z)| \le \frac{C}{1+|z|},
\quad \mu\le |\arg z|\le\pi.
\end{equation}
\end{lem}
The proof can be found in Gorenflo and Mainardi \cite{GM08}, or in Podlubny \cite{P99} on p.35.

In view of $\{\lambda_k,\varphi_k\}_{k=1}^\infty$, the solution $u$ to \eqref{equ-u} and \eqref{equ-ibc}   
can be rewritten as follows: 
\begin{equation}\label{sol-u}
\begin{aligned}
u(t) =&\; e^{-tA}u_0 + \int_0^t e^{-(t-s)A} f(s) ds \\
&+ \int_0^t e^{-(t-s)A}\left(c(s)u(s)-q(s) \partial_s^\alpha u(s)\right) ds
\end{aligned}
\end{equation}
where the operator $e^{-tA}$, $t\ge0$, is defined by 
\begin{equation}\label{def-e^tL}
e^{-tA}g := 
\sum_{k=1}^\infty e^{-\lambda_k t}\langle g,\varphi_k \rangle\varphi_k, \quad g\in H^{-1}(\Omega).
\end{equation}
We denote
$$
F(t) := e^{-tA}u_0 + \int_0^t e^{-(t-s)A} f(s) ds,
$$
and
\begin{equation}
\label{def-K}
Ku(t) := \int_0^t e^{-(t-s)A}\left(c(s)u(s)-q(s) \partial_s^\alpha u(s)\right) ds,\quad u\in D(K),
\end{equation}
where $D(K):= H^{\alpha_1} (0,T;H^{-1}(\Omega))$ with $1>\alpha_1>\max\{\alpha,\frac12\}$. 
From \eqref{sol-u}, we obtain
\begin{equation}\label{equ-Ku}
u(t) = F(t) + Ku(t).
\end{equation}
Here and henceforth, $\partial_t^\alpha$ means the fractional derivative whose domain is extended to 
the fractional Sobolev space $H^{\alpha_1}(0,T)$. 
For the detailed descriptions, we refer to Gorenflo, Luchko and Yamamoto \cite{GLY15} and 
Kubica and Yamamoto \cite{KY20}.

\subsection{Unique existence}
In this subsection, we shall prove the unique existence of solution in $H^{\alpha_1}(0,T; V)$ with 
$V=H^{-1}(\Omega)$. 
For the case where $V=L^2(\Omega)$, we can apply similar argument as follows. 
Thus, for the sake of simplicity, we omit the proof of the case $V=L^2(\Omega)$ in this paper.

According to the regularity assumptions on the coefficients $c,q$, 
we see that $cu-q\partial_t^\alpha u \in L^2(0,T; H^{-1}(\Omega))$ provided $u\in D(K)$. 
It is readily to check that $Ku$ is the solution to the following parabolic equation: 
\begin{equation}\label{equ-para}
\begin{cases}
\!\begin{alignedat}{2}
& \partial_t Ku + A Ku = cu - q\partial_t^\alpha u, 
&\quad & \mbox{ in }\Omega\times(0,T),\\
& Ku = 0, &\quad & \mbox{ on }\partial\Omega\times (0,T), \\
& Ku(\cdot,0) = 0, &\quad & \mbox{ in }\Omega,
\end{alignedat}
\end{cases}
\end{equation}
and then by the well-known regularity for parabolic equations 
(e.g., \cite[Section~{\bf 4.7.1}, p.243]{LM1}), 
we have $Ku\in H^1(0,T; H^{-1}(\Omega))\cap L^2(0,T; H_0^1(\Omega))$. 
By Theorem~{\bf 2.1} in \cite{T79} and Theorem~{\bf 16.2}, Chapter~{\bf 1} in \cite{LM1}, 
we find that $H^1(0,T; H^{-1}(\Omega))\cap L^2(0,T; L^2(\Omega))$ is compact in 
$H^{\alpha_1}(0,T; H^{-1}(\Omega))$, 
which implies $K: H^{\alpha_1}(0,T; H^{-1}(\Omega)) \rightarrow H^{\alpha_1}(0,T; H^{-1}(\Omega))$ 
is a compact operator. 
By the Fredholm alternative, \eqref{equ-Ku} admits a unique solution in 
$H^{\alpha_1}(0,T; H^{-1}(\Omega))$ as long as 
$$
\mbox{(\romannumeral1)}\quad F\in H^{\alpha_1}(0,T; H^{-1}(\Omega)), \hspace{4.7cm}
$$  
$$
\mbox{(\romannumeral2)}\quad I-K \mbox{ is one-to-one on } H^{\alpha_1}(0,T; H^{-1}(\Omega)), 
\mbox{ that is,} 
$$
$$
(I-K)v = 0 \ \mbox{ implies } \ v=0, \hspace{3cm}
$$
are valid. 
Here $I$ denotes the identity operator. Noting that $F$ is the solution to 
\begin{equation}\label{equ-F}
\begin{cases}
\!\begin{alignedat}{2}
& \partial_t F(x,t) + A F(x,t) = f(x,t), 
&\quad & (x,t)\in \Omega\times (0,T),\\
& F(x,t) = 0, &\quad & (x,t)\in \partial\Omega\times (0,T), \\
& F(x,0) = u_0, &\quad & x\in \Omega,
\end{alignedat}
\end{cases}
\end{equation}
by the regularity assumptions $u_0\in L^2(\Omega)$, $f\in L^2(0,T; H^{-1}(\Omega))$, 
we have $F\in H^1(0,T; H^{-1}(\Omega)) \subset H^{\alpha_1}(0,T; H^{-1}(\Omega))$. 
Thus, (\romannumeral1) is verified. 

Next we check (\romannumeral2). Indeed we show the following uniqueness result.
\begin{lem}
\label{lem-K}
Assume $v\in D(K)$ satisfies the following integral equation:
$$
v=Kv. 
$$
where the operator $K$ is defined in \eqref{def-K}. Then 
$$
v=0.
$$
\end{lem}
To prove this result, we need several lemmata.
\begin{lem}
\label{lem-exp}
Let $0\le\beta<1$, $0\le s\le t$. Then
\begin{equation}
\label{eq-e}
\int_s^t \frac{(t-\tau)^{-\beta}}{\Gamma(1-\beta)} e^{-\lambda(\tau-s)} d\tau 
= (t-s)^{1-\beta} E_{1,2-\beta} (-\lambda(t-s)).
\end{equation}
\end{lem}
\begin{proof}
From the series expansion of the exponential function, by a direct calculation, we find
\begin{align*}
\int_s^t \frac{(t-\tau)^{-\beta}}{\Gamma(1-\beta)} e^{-\lambda(\tau-s)} d\tau 
&= \int_s^t \frac{(t-\tau)^{-\beta}}{\Gamma(1-\beta)} \sum_{n=0}^\infty \frac{(-\lambda(\tau-s))^n}{n!} 
d\tau\\
&= \sum_{n=0}^\infty \frac{(-\lambda)^n}{n!\Gamma(1-\beta)} 
\int_s^t (t-\tau)^{-\beta} (\tau-s)^n d\tau\\
&= \sum_{n=0}^\infty \frac{(-\lambda)^n}{n!\Gamma(1-\beta)} 
\int_0^{t-s} (t-s-\tau)^{-\beta} \tau^n d\tau\\
&= \sum_{n=0}^\infty \frac{(-\lambda)^n}{n!\Gamma(1-\beta)} (t-s)^{n+1-\beta} B(1-\beta, n+1),
\end{align*}
where $B(a,b)$ denotes the beta function. 
Moreover, noting the identity between the beta function and the gamma function: 
$B(a,b) = \frac{\Gamma(a) \Gamma(b)}{\Gamma(a+b)}$ and $\Gamma(n+1) = n!$, 
from the definition of the Mittag-Leffler function we obtain
\begin{align*}
\int_s^t \frac{(t-\tau)^{-\beta}}{\Gamma(1-\beta)} e^{-\lambda(\tau-s)} d\tau 
= (t-s)^{1-\beta} E_{1,2-\beta} (-\lambda(t-s)),
\end{align*}
which completes the proof of the lemma.
\end{proof}

On the basis of the above lemma, we further have
\begin{lem}
Let $0\le \beta<1$. Then there exists a constant $C=C(\beta)>0$ such that the following inequality 
\begin{equation}
\label{esti-e}
\left\|\partial_t^\beta \int_0^t e^{-(t-s)A} w(s)ds\right\|_{H^{-1}(\Omega)} 
\le C\int_0^t (t-s)^{-\beta}\|w(s)\|_{H^{-1}(\Omega)} ds
\end{equation}
holds true for any $w\in L^2(0,T; H^{-1}(\Omega))$.
\end{lem}
\begin{proof}
By \eqref{def-e^tL}, we divide $\partial_t^\beta \int_0^t e^{-(t-s)A} w(s)ds$ into two parts: 
\begin{align*}
& I_1 := \frac{1}{\Gamma(1-\beta)}\int_0^t (t-s)^{-\beta} w(s) ds,\\
& I_2 := \frac{1}{\Gamma(1-\beta)}\int_0^t (t-\tau)^{-\beta} \int_0^\tau \partial_\tau 
\left(\sum_{k=1}^\infty e^{-\lambda_k (\tau-s)}\langle w(s),\varphi_k \rangle\varphi_k \right)ds d\tau.
\end{align*}
For any $\psi\in H_0^1(\Omega)$, 
\begin{align*}
\left|\langle I_1,\psi \rangle\right| 
&=\left|\frac{1}{\Gamma(1-\beta)}\int_0^t (t-s)^{-\beta} \langle w(s),\psi \rangle ds\right|\\
&\le \frac{1}{\Gamma(1-\beta)}\int_0^t (t-s)^{-\beta} |\langle w(s),\psi \rangle| ds.
\end{align*}
Thus, we have 
\begin{align*}
\|I_1\|_{H^{-1}(\Omega)} &= \sup_{\|\psi\|_{H_0^1(\Omega)}=1} \left|\langle I_1,\psi \rangle\right|\\
&\le \frac{1}{\Gamma(1-\beta)}\int_0^t (t-s)^{-\beta} \|w(s)\|_{H^{-1}(\Omega)} ds.
\end{align*}
On the other hand, by Fubini's theorem, noting the identity \eqref{eq-e}, we calculate
\begin{align*}
I_2 
&= \sum_{k=1}^\infty -\lambda_k \int_0^t \langle w(s),\varphi_k \rangle\varphi_k 
\int_s^t \frac{(t-\tau)^{-\beta}}{\Gamma(1-\beta)} e^{-\lambda_k(\tau-s)} d\tau ds\\
&= \sum_{k=1}^\infty -\lambda_k \int_0^t (t-s)^{1-\beta} E_{1,2-\beta} (-\lambda_k(t-s)) 
\langle w(s),\varphi_k \rangle \varphi_k ds.
\end{align*}
Consequently, for any $\psi\in H_0^1(\Omega)$, we use the estimate \eqref{esti-ML} for 
the Mittag-Leffler functions to derive
\begin{align*}
\left|\langle I_2,\psi \rangle\right| 
&\le \sum_{k=1}^\infty \int_0^t \lambda_k (t-s)^{1-\beta} \left|E_{1,2-\beta} (-\lambda_k (t-s))\right| 
\left|\langle w(s),\varphi_k \rangle\right| \left|\langle \varphi_k,\psi \rangle\right| ds\\
&\le C\sum_{k=1}^\infty \int_0^t (t-s)^{-\beta} \frac{\lambda_k (t-s)}{1+\lambda_k (t-s)} 
\left|\lambda_k^{-\frac12}\langle w(s),\varphi_k \rangle\right| 
\left|\lambda_k^{\frac12}\langle \varphi_k,\psi \rangle\right| ds,
\end{align*}
which combined with H\"older's inequality implies
\begin{align*}
\left|\langle I_2,\psi \rangle\right| 
&\le C\int_0^t (t-s)^{-\beta} 
\left(\sum_{k=1}^\infty \lambda_k^{-1}\left|\langle w(s),\varphi_k \rangle\right|^2 \right)^{\frac12}
\left(\sum_{k=1}^\infty \lambda_k \left|\langle \varphi_k,\psi \rangle\right|^2 \right)^{\frac12} ds\\
&= C\|\psi\|_{H_0^1(\Omega)} \int_0^t (t-s)^{-\beta} \|w(s)\|_{H^{-1}(\Omega)} ds.
\end{align*}
Finally, we have 
\begin{align*}
\|I_2\|_{H^{-1}(\Omega)} \le C\int_0^t (t-s)^{-\beta} \|w(s)\|_{H^{-1}(\Omega)} ds.
\end{align*}
Thus, we complete the proof by the triangle inequality for the norm.
\end{proof}

Now we give the proof of Lemma~{\bf 3.3}. 
\begin{proof}[Proof of Lemma~{\bf 3.3}]
According to the equation $v=Kv$, we find
\begin{align*}
\left\|\partial_t^\beta v(t)\right\|_{H^{-1}(\Omega)} 
&= \left\|\partial_t^\beta Kv(t)\right\|_{H^{-1}(\Omega)} \\
&= \left\|\partial_t^\beta \int_0^t e^{-(t-s)A} (c(s)v(s)-q(s)\partial_s^\alpha v(s))ds 
\right\|_{H^{-1}(\Omega)} 
\end{align*}
for $0\le \beta <1$. 
By taking $\beta = \alpha_1$, $\beta = 0$ in the estimate \eqref{esti-e} separately, and noting that 
$c\in L^\infty(0,T;W^{2,\infty}(\Omega))$, $q\in L^\infty(0,T)$, we obtain
\begin{align*}
\|\partial_t^{\alpha_1} v(t)\|_{H^{-1}(\Omega)} 
&\le C\int_0^t (t-s)^{-\alpha_1}\|v(s)\|_{H^{-1}(\Omega)} ds \\
&\quad + C\int_0^t (t-s)^{-\alpha_1}\|\partial_s^\alpha v(s)\|_{H^{-1}(\Omega)} ds 
\end{align*}
and
\begin{align*}
\|v(t)\|_{H^{-1}(\Omega)} 
\le C\int_0^t \|v(s)\|_{H^{-1}(\Omega)} ds + C\int_0^t \|\partial_s^\alpha v(s)\|_{H^{-1}(\Omega)} ds. 
\end{align*}
Moreover, by noting the semigroup property 
$J^{\gamma_1+\gamma_2}=J^{\gamma_1}J^{\gamma_2}$, $\gamma_1,\gamma_2>0$ 
of the Riemann-Liouville fractional integral operator which is defined by
$$
J^\gamma g(t) := \frac{1}{\Gamma(\gamma)}\int_0^t (t-\tau)^{\gamma-1} g(\tau) d\tau,
\quad \gamma > 0,
$$
we see that
$$
\|\partial_s^\alpha v(s)\|_{H^{-1}(\Omega)} 
\le C J^{\alpha_1-\alpha}\|\partial_t^{\alpha_1} v(s)\|_{H^{-1}(\Omega)},
$$
from which we further obtain that
\begin{align*}
\int_0^t \|\partial_s^\alpha v(s)\|_{H^{-1}(\Omega)} ds
&\le CJ^{1+\alpha_1-\alpha} \|\partial_t^{\alpha_1} v(t)\|_{H^{-1}(\Omega)}\\
&\le C\int_0^t (t-s)^{\alpha_1-\alpha}  \|\partial_t^{\alpha_1} v(s)\|_{H^{-1}(\Omega)} ds,
\end{align*}
and 
\begin{align*}
\int_0^t (t-s)^{-\alpha_1}\|\partial_s^\alpha v(s)\|_{H^{-1}(\Omega)} ds
&\le CJ^{1-\alpha} \|\partial_t^{\alpha_1} v(t)\|_{H^{-1}(\Omega)}\\
&\le C\int_0^t (t-s)^{-\alpha}  \|\partial_t^{\alpha_1} v(s)\|_{H^{-1}(\Omega)} ds.
\end{align*}
Finally, since $1, (t-s)^{-\alpha}, (t-s)^{\alpha_1-\alpha}\le C(t-s)^{-\alpha_1}$ for $\alpha_1>\alpha$, 
we obtain
\begin{align*}
&\|v(t)\|_{H^{-1}(\Omega)} + \|\partial_t^{\alpha_1} v(t)\|_{H^{-1}(\Omega)} \\
\le &\; C\int_0^t (t-s)^{-\alpha_1}
\left(\|v(s)\|_{H^{-1}(\Omega)} + \|\partial_s^{\alpha_1} v(s)\|_{H^{-1}(\Omega)}\right) ds
\end{align*}
with a generic constant $C>0$ which depends also on $T$. 
Therefore, generalized Gr\"onwall's inequality (see e.g., \cite[lemma~{\bf 7.1.1}]{H81}) implies $v=0$. 
We finish the proof of the lemma. 
\end{proof}

By the Fredholm alternative, we proved that the initial-boundary value problem 
\eqref{equ-u}--\eqref{equ-ibc} admits a unique solution in $H^{\alpha_1}(0,T; H^{-1}(\Omega))$ 
with $1>\alpha_1>\max\{\frac12,\alpha\}$. 

\subsection{Improved regularity}
Next we show the improved regularity and some estimates by using the integral form \eqref{sol-u}. 

Recalling that we rewrite \eqref{sol-u} by \eqref{equ-Ku}, 
then it is sufficient to discuss the regularity for $F$ and $Ku$, respectively. 
Since $F$ is the solution to parabolic equation \eqref{equ-F}, 
under the assumptions that $u_0\in L^2(\Omega)$, $f\in L^2(0,T;H^{-1}(\Omega))$, 
we obtain by classical regularity for parabolic equations 
that $F\in H^1(0,T; H^{-1}(\Omega))\cap L^2(0,T;H_0^1(\Omega)) \cap C([0,T]; L^2(\Omega))$ 
(e.g., \cite[Example~{\bf 4.7.1}, Chapter~{\bf 3}]{LM1}). 
Similarly, we have the same regularity for $Ku$. Therefore, we find the improved regularity 
$$
u\in H^1(0,T; H^{-1}(\Omega))\cap L^2(0,T;H_0^1(\Omega)) \cap C([0,T];L^2(\Omega)).
$$

In order to establish the estimate for the solution $u$, we need the following lemmata.
\begin{lem}
\label{lem-Ku}
There exists a constant $C=C(\alpha, c,q,T)>0$ such that
\begin{align*}
\|\partial_t Ku(t)\|_{H^{-1}(\Omega)} 
&\le C\|u_0\|_{L^2(\Omega)} \\
&\quad + C\int_0^t (t-\tau)^{-\alpha} \left(\|u(\tau)\|_{H_0^1(\Omega)} 
+ \|\partial_\tau u(\tau)\|_{H^{-1}(\Omega)}\right) d\tau.
\end{align*}
\end{lem} 
\begin{proof}
We divide $\partial_t Ku(t)$ into three parts: 
\begin{align*}
&I_0 := c(t)u(t) - q(t)\partial_t^\alpha u(t),\\
&I_1 := \int_0^t \sum_{k=1}^\infty \partial_t \left(e^{-\lambda_k(t-s)} 
\langle c(s)u(s),\varphi_k \rangle \varphi_k \right) ds,\\
&I_2 := \int_0^t \sum_{k=1}^\infty \partial_t \left(e^{-\lambda_k(t-s)} 
\langle q(s)\partial_s^\alpha u(s),\varphi_k \rangle \varphi_k \right) ds.
\end{align*}
Since $c\in L^\infty(0,T; W^{2,\infty}(\Omega))$, $q\in L^\infty(0,T)$, it follows that
\begin{align*}
\|I_0\|_{H^{-1}(\Omega)} 
&\le C\left\|u_0 + \int_0^t \partial_\tau u(\tau) d\tau\right\|_{H^{-1}(\Omega)} 
+ C\left\|\int_0^t (t-\tau)^{-\alpha} \partial_\tau u(\tau) d\tau\right\|_{H^{-1}(\Omega)}\\
&\le C\left(\|u_0\|_{L^2(\Omega)} + \int_0^t (t-\tau)^{-\alpha} \|\partial_\tau u(\tau)\|_{H^{-1}(\Omega)} 
d\tau\right).
\end{align*}
Here we used the triangle inequality and $u(t) = u_0 + \int_0^t \partial_\tau u(\tau) d\tau$.
Next we derive the estimations for $I_1$ and $I_2$. In fact, for any $\psi\in H_0^1(\Omega)$, 
we conclude from H\"older's inequality that
\begin{align*}
|\langle I_1,\psi \rangle | 
&\le \sum_{k=1}^\infty \int_0^t \lambda_k e^{-\lambda_k(t-s)}
\left|\langle c(s)u(s),\varphi_k \rangle \right| \left|\langle \varphi_k,\psi \rangle \right| ds\\
&\le C\int_0^t \left(\sum_{k=1}^\infty  \lambda_k \left|\langle \varphi_k,\psi \rangle \right|^2 
ds\right)^{\frac12} \left(\sum_{k=1}^\infty\lambda_k e^{-2\lambda_k(t-s)} 
\left|\langle c(s)u(s),\varphi_k \rangle \right|^2 \right)^{\frac12}ds\\
&\le C\|\psi\|_{H_0^1(\Omega)} \int_0^t\left(\sum_{k=1}^\infty \lambda_k e^{-2\lambda_k(t-s)}
\left|\langle c(s)u(s),\varphi_k \rangle \right|^2 \right)^{\frac12}ds.
\end{align*}
Moreover, noting that $e^{-2\lambda_k s} \le 1$ for $s>0$, we see that
\begin{align*}
|\langle I_1,\psi \rangle| &\le  C\|\psi\|_{H_0^1(\Omega)} \int_0^t \left(\sum_{k=1}^\infty 
\lambda_k \left|\langle c(s)u(s),\varphi_k \rangle \right|^2 \right)^{\frac12} ds\\
&\le  C\|\psi\|_{H_0^1(\Omega)} \int_0^t \|c(s)u(s)\|_{H_0^1(\Omega)} ds,
\end{align*}
which combined with the assumption that $c\in L^\infty(0,T; W^{2,\infty}(\Omega))$ implies
\begin{align*}
|\langle I_1,\psi \rangle | 
&\le C\|\psi\|_{H_0^1(\Omega)} \int_0^t \|u(s)\|_{H_0^1(\Omega)} ds,
\end{align*}
that is,
\begin{align*}
\|I_1\|_{H^{-1}(\Omega)} 
\le C \int_0^t \|u(s)\|_{H_0^1(\Omega)}  ds .
\end{align*}
On the other hand, for any $\psi\in H_0^1(\Omega)$, we have
\begin{align*}
| \langle I_2,\psi \rangle |
&= \left| -\sum_{k=1}^\infty \int_0^t \lambda_k e^{-\lambda_k(t-s)} q(s)
\langle \partial_s^\alpha u(s),\varphi_k \rangle \langle \varphi_k,\psi \rangle ds \right|\\
&\le \sum_{k=1}^\infty \int_0^t \lambda_k e^{-\lambda_k(t-s)} \left|\langle \varphi_k,\psi \rangle \right|
|q(s)| \int_0^s \frac{(s-\tau)^{-\alpha}}{\Gamma(1-\alpha)} 
\left|\langle \partial_\tau u(\tau),\varphi_k \rangle \right| d\tau ds\\
&\le C\sum_{k=1}^\infty \lambda_k \left|\langle \varphi_k,\psi \rangle \right| 
\int_0^t \left|\langle \partial_\tau u(\tau),\varphi_k \rangle \right|
\int_\tau^t \frac{(s-\tau)^{-\alpha}}{\Gamma(1-\alpha)} e^{-\lambda_k(t-s)} ds d\tau. 
\end{align*}
Here the last equality is due to Fubini's theorem. 
Similarly to the proof of Lemma~{\bf 3.5}, we obtain
\begin{align*}
\|I_2\|_{H^{-1}(\Omega)} 
&\le C\int_0^t (t-\tau)^{-\alpha} \|\partial_\tau u(\tau)\|_{H^{-1}(\Omega)} d\tau.  
\end{align*}
Collecting all the above estimates, and noting that 
$$
1 \le T^{\alpha} t^{-\alpha} \le Ct^{-\alpha},\quad 0<t\le T,
$$
we finish the proof of the lemma.
\end{proof}

In a similar way, we can prove
\begin{lem}
\label{lem-Ku1}
There exists a constant $C=C(\alpha,c,q,T)>0$ such that
\begin{align*}
\|Ku(t)\|_{H_0^1(\Omega)} 
&\le C\int_0^t (t-\tau)^{-\alpha} \left(\|u(\tau)\|_{H_0^1(\Omega)} 
+ \|\partial_\tau u(\tau)\|_{H^{-1}(\Omega)}\right) d\tau.
\end{align*}
\end{lem} 
Then by Lemma~{\bf 3.6}, Lemma~{\bf 3.7} and \eqref{equ-Ku}, we obtain
\begin{align*}
v(t) &\le a(t) + C\int_0^t (t-s)^{-\alpha} v(s) ds 
\end{align*}
where 
\begin{align*}
v(t) &= \|\partial_t u(t)\|_{H^{-1}(\Omega)} + \|u(t)\|_{H_0^1(\Omega)},\\
a(t) &= C\|u_0\|_{L^2(\Omega)} + \|\partial_t F(t)\|_{H^{-1}(\Omega)} + \|F(t)\|_{H_0^1(\Omega)}.
\end{align*}
Here the generic constant $C>0$ is independent of $t$, but may depend on $\alpha$ and $T$ as well. 
Finally, we employ the following generalized Gr\"onwall's inequality from \cite[Lemma~{\bf 7.1.1}]{H81}. 
\begin{lem}
\label{lem-gronwall}
Suppose $b\ge 0, \beta>0$ and $a(t)$ is a nonnegative function locally integrable on $0\le t < T$, 
and suppose $v(t)$ is nonnegative and locally integrable on $0\le t < T$ with
$$
v(t) \le a(t) + b\int_0^t (t-s)^{\beta-1} v(s) ds
$$
on this interval. Then 
$$
v(t) \le a(t) + b\Gamma(\beta)\int_0^t (t-s)^{\beta-1} E_{\beta,\beta}(b\Gamma(\beta) (t-s)^\beta) a(s) 
ds,\quad 0\le t < T.
$$
In particular, there exists a constant $C=C(b,\beta,T)>0$ such that
$$
v(t) \le a(t) + C\int_0^t (t-s)^{\beta-1} a(s) ds,\quad 0\le t < T.
$$
\end{lem}

Now we are ready to establish the estimates in Theorem \ref{thm-uniq-exist}.
By Lemma~{\bf 3.8}, we have
\begin{align*}
v(t)
&\le C\|u_0\|_{L^2(\Omega)} + \|\partial_t F(t)\|_{H^{-1}(\Omega)} + \|F(t)\|_{H_0^1(\Omega)}\\
&\quad + C\int_0^t (t-s)^{-\alpha} (\|u_0\|_{L^2(\Omega)} + \|\partial_s F(s)\|_{H^{-1}(\Omega)} 
+ \|F(s)\|_{H_0^1(\Omega)}) ds
\end{align*}
with a new generic constant $C>0$. 
We take $L^2$-norm over $t\in (0,T)$ on both sides and by Young's convolution inequality, we obtain
$$
\|v\|_{L^2(0,T)} \le C\left(\|u_0\|_{L^2(\Omega)} + \|\partial_t F\|_{L^2(0,T;H^{-1}(\Omega))} 
+ \|F\|_{L^2(0,T;H_0^1(\Omega))}\right). 
$$
We complete the first statement \eqref{esti-fst} of Theorem \ref{thm-uniq-exist} by noting that 
the following regularity estimate 
$$
\|\partial_t F\|_{L^2(0,T;H^{-1}(\Omega))} + \|F\|_{L^2(0,T;H_0^1(\Omega))}
\le C\left(\|u_0\|_{L^2(\Omega)} + \|f\|_{L^2(0,T; H^{-1}(\Omega))}\right)
$$
is valid since $F$ is the solution to parabolic problem \eqref{equ-F} with 
$f\in L^2(0,T; H^{-1}(\Omega))$ and $u_0\in L^2(\Omega)$. 

For the second statement \eqref{esti-sec}, recall that we assume 
$u_0\in H_0^1(\Omega), f\in L^2(0,T; L^2(\Omega))$ and $c\in L^\infty(0,T; W^{2,\infty}(\Omega))$. 
By the well-known regularity for parabolic equations (e.g., \cite[Chapter 7]{E98}), 
it is readily to see that 
$$
u = Ku + F \in H^1(0,T; L^2(\Omega)) \cap L^2(0,T; H^2(\Omega)\cap H_0^1(\Omega)) 
\cap C([0,T]; H_0^1(\Omega)). 
$$
In a similar way, the second regularity estimate \eqref{esti-sec} follows immediately from
$$
\|\partial_t F\|_{L^2(0,T;L^2(\Omega))} + \|F\|_{L^2(0,T;H^2(\Omega))}
\le C\left(\|u_0\|_{H_0^1(\Omega)} + \|f\|_{L^2(0,T; L^2(\Omega))}\right),
$$
the generalized Gr\"onwall inequality (Lemma~{\bf 3.8}) and the next lemma: 
\begin{lem}
\label{lem-Ku2}
There exists a constant $C=C(\alpha, c,q,T)>0$ such that
\begin{align*}
\|\partial_t Kv(t)\|_{L^2(\Omega)}^2 
&\le C\|v(0)\|_{L^2(\Omega)}^2 + C\int_0^t \|v(\tau)\|_{H^2(\Omega)}^2 d\tau\\
&\quad + C\int_0^t (t-\tau)^{-\alpha} \|\partial_\tau v(\tau)\|_{L^2(\Omega)}^2 d\tau,
\end{align*}
and
\begin{align*}
\|Kv(t)\|_{H^2(\Omega)}^2 
&\le C\int_0^t \|v(\tau)\|_{H^2(\Omega)}^2 d\tau 
+ C\int_0^t (t-\tau)^{-\alpha} \|\partial_\tau v(\tau)\|_{L^2(\Omega)}^2 d\tau,
\end{align*}
for all $v\in H^1(0,T;L^2(\Omega))\cap L^2(0,T; H^2(\Omega)\cap H_0^1(\Omega))$.
\end{lem} 
Here we omit the proof of the above lemma since it is similar to those of Lemmata~{\bf 3.6, 3.7} 
while we note the equivalence of norms: 
$$
\|w\|_{H^2(\Omega)}^2 \doteqdot \|w\|_{H^2(\Omega)\cap H_0^1(\Omega)}^2 := 
\sum_{k=1}^\infty \lambda_k^2 |(w,\varphi_k)|^2
$$
for $w\in H^2(\Omega)\cap H_0^1(\Omega)$. 

Finally, we further assume $u_0\in H^2(\Omega)\cap H_0^1(\Omega)$, $f\in H^1(0,T; L^2(\Omega))$ 
and prove the third statement of Theorem~{\ref{thm-uniq-exist}}. 

By iterations and \eqref{equ-Ku}, we obtain
\begin{equation}
\label{equ-Ku1}
u = K^N u + \sum_{j=0}^{N-1} K^j F
\end{equation}
for some large number $N\ge \frac{1}{1-\alpha}$. 
In order to argue the regularity of solution $u$, it is sufficient to deal with $K^N u$ and $K^j F$, 
$j=0,1,\ldots,N-1$, respectively. 
In terms of Lemma~{\bf 3.9}, we obtain 
\begin{equation}
\label{eq-Ku1}
\begin{aligned}
\|\partial_t K^j u(t)\|_{L^2(\Omega)}^2 
&\le C\|K^{j-1}u(0)\|_{L^2(\Omega)}^2 + C\int_0^t \|K^{j-1}u(\tau)\|_{H^2(\Omega)}^2 d\tau\\
&\quad + C\int_0^t (t-\tau)^{-\alpha} \|\partial_\tau K^{j-1}u(\tau)\|_{L^2(\Omega)}^2 d\tau,
\end{aligned}
\end{equation}
and
\begin{equation}
\label{eq-Ku2}
\begin{aligned}
\|K^j u(t)\|_{H^2(\Omega)}^2 
&\le C\int_0^t \|K^{j-1}u(\tau)\|_{H^2(\Omega)}^2 d\tau\\
&\quad + C\int_0^t (t-\tau)^{-\alpha} \|\partial_\tau K^{j-1}u(\tau)\|_{L^2(\Omega)}^2 d\tau
\end{aligned}
\end{equation}
for $j=1,2,\ldots,N$. 
By the definition of operator $K$, it is readily to see that $K^{j-1} u(0) = 0$, $j=2,\ldots,N$. 
Then by \eqref{eq-Ku1}--\eqref{eq-Ku2} and $(t-\tau)^{\alpha} \le T^\alpha$, we obtain
\begin{equation}
\label{eq-Ku3}
\begin{aligned}
&\|\partial_t K^j u(t)\|_{L^2(\Omega)}^2 + \|K^j u(t)\|_{H^2(\Omega)}^2\\
\le &\; C\int_0^t (t-\tau)^{-\alpha} \left(\|\partial_\tau K^{j-1}u(\tau)\|_{L^2(\Omega)}^2 
+ \|K^{j-1}u(\tau)\|_{H^2(\Omega)}^2 \right)d\tau
\end{aligned}
\end{equation}
for all $j=2,3,\ldots,N$. 
Recall that $C>0$ denotes a generic constant which means $C$ can change values in different lines. 
By using \eqref{eq-Ku3} with $j=N, N-1$ and by a direct calculation, 
we arrive at the following estimate: 
\begin{align*}
&\|\partial_t K^N u(t)\|_{L^2(\Omega)}^2 + \|K^N u(t)\|_{H^2(\Omega)}^2\\
\le&\; C\int_0^t (t-\tau)^{-\alpha} \left(\|\partial_\tau K^{N-1}u(\tau)\|_{L^2(\Omega)}^2 
+ \|K^{N-1}u(\tau)\|_{H^2(\Omega)}^2 \right)d\tau\\
\le&\; C\int_0^t (t-\tau)^{-\alpha} \int_0^\tau (\tau-s)^{-\alpha} 
\left(\|\partial_s K^{N-2}u(s)\|_{L^2(\Omega)}^2 + \|K^{N-2}u(s)\|_{H^2(\Omega)}^2 \right) ds d\tau. 
\end{align*}
Moreover, by Fubini's theorem, we see that
\begin{align*}
&\|\partial_t K^N u(t)\|_{L^2(\Omega)}^2 + \|K^N u(t)\|_{H^2(\Omega)}^2\\
\le&\; C\int_0^t \left(\|\partial_s K^{N-2}u(s)\|_{L^2(\Omega)}^2 
+ \|K^{N-2}u(s)\|_{H^2(\Omega)}^2 \right) \int_s^t (t-\tau)^{-\alpha} (\tau-s)^{-\alpha} d\tau ds\\
\le&\; C\int_0^t (t-s)^{1-2\alpha} \left(\|\partial_s K^{N-2}u(s)\|_{L^2(\Omega)}^2 
+ \|K^{N-2}u(s)\|_{H^2(\Omega)}^2 \right) ds. 
\end{align*}

We calculate by iterations and obtain
\begin{align*}
&\|\partial_t K^N u(t)\|_{L^2(\Omega)}^2 + \|K^N u(t)\|_{H^2(\Omega)}^2\\
\le&\; C\int_0^t (t-s)^{-\alpha + (N-2)(1-\alpha)} \left(\|\partial_s Ku(s)\|_{L^2(\Omega)}^2 
+ \|Ku(s)\|_{H^2(\Omega)}^2 \right) ds\\
\le&\; C\int_0^t (t-s)^{-\alpha + (N-2)(1-\alpha)} \int_0^s (s-\tau)^{-\alpha} 
\left( \|\partial_\tau u(\tau)\|_{L^2(\Omega)}^2 + \|u(\tau)\|_{H^2(\Omega)}^2 \right) d\tau ds\\
& + C\int_0^t (t-s)^{-\alpha + (N-2)(1-\alpha)} \|u_0\|_{L^2(\Omega)}^2 ds\\
\le&\; C\|u_0\|_{L^2(\Omega)}^2 + C\int_0^t (t-\tau)^{N(1-\alpha)-1} 
\left(\|\partial_\tau u(\tau)\|_{L^2(\Omega)}^2 + \|u(\tau)\|_{H^2(\Omega)}^2 \right) d\tau.
\end{align*}
Since $N\ge \frac{1}{1-\alpha}$ implies $N(1-\alpha) -1 \ge 0$, the above inequality yields
\begin{equation}
\label{eq-Ku4}
\begin{aligned}
&\|\partial_t K^N u(t)\|_{L^2(\Omega)}^2 + \|K^N u(t)\|_{H^2(\Omega)}^2\\
\le&\; C\|u_0\|_{L^2(\Omega)}^2 + C\int_0^t \left(\|\partial_\tau u(\tau)\|_{L^2(\Omega)}^2 
+ \|u(\tau)\|_{H^2(\Omega)}^2 \right) d\tau.
\end{aligned}
\end{equation}
Noting that we have proved 
$u\in H^1(0,T;L^2(\Omega))\cap L^2(0,T; H^2(\Omega)\cap H_0^1(\Omega))$, 
then for $t\in (0,T)$, the right-hand side of \eqref{eq-Ku4} is finite, which leads to 
$$
K^N u \in W^{1,\infty}(0,T;L^2(\Omega)) \cap L^\infty(0,T;H^2(\Omega)\cap H_0^1(\Omega)). 
$$
In the same way, we can prove
\begin{align*}
&\|\partial_t K^j F(t)\|_{L^2(\Omega)}^2 + \|K^j F(t)\|_{H^2(\Omega)}^2\\
\le&\; C\|F(0)\|_{L^2(\Omega)}^2 + C\int_0^t (t-\tau)^{j(1-\alpha)-1} 
\left(\|\partial_\tau F(\tau)\|_{L^2(\Omega)}^2 + \|F(\tau)\|_{H^2(\Omega)}^2 \right) d\tau
\end{align*}
for all $j=1,2,\ldots,N-1$. 
Moreover, under the assumptions 
$u_0\in H^2(\Omega)\cap H_0^1(\Omega)$, $f\in H^1(0,T;L^2(\Omega))$, 
improved regularity for parabolic equations (e.g., \cite{E98}) yields that
$F\in W^{1,\infty}(0,T;L^2(\Omega))\cap L^\infty(0,T;H^2(\Omega)\cap H_0^1(\Omega))$ 
and the regularity estimate 
\begin{equation*}
\|\partial_t F(t)\|_{L^2(\Omega)}^2 + \|F(t)\|_{H^2(\Omega)}^2 
\le C \left( \|u_0\|_{H^2(\Omega)}^2 + \|f\|_{H^1(0,T;L^2(\Omega))}^2\right) 
\end{equation*}
holds true for $t\in(0,T)$, which further implies that
\begin{align*}
&\|\partial_t K^j F(t)\|_{L^2(\Omega)}^2 + \|K^j F(t)\|_{H^2(\Omega)}^2\\
\le&\; C\|u_0\|_{L^2(\Omega)}^2 + C\int_0^t (t-\tau)^{-\alpha} 
\left(\|\partial_\tau F(\tau)\|_{L^2(\Omega)}^2 + \|F(\tau)\|_{H^2(\Omega)}^2 \right) d\tau \\
\le&\; C(1 + t^{1-\alpha}) \left(\|u_0\|_{H^2(\Omega)}^2 + \|f\|_{H^1(0,T;L^2(\Omega))}^2 \right).
\end{align*}
Here in the first inequality we used $F(0) = u_0$ from \eqref{equ-F} and the estimate that 
$(t-\tau)^{(j-1)(1-\alpha)} \le T^{(j-1)(1-\alpha)}\le T^{(N-2)(1-\alpha)}$ for $j=1,2,\cdots,N-1$. 
We finally see that
\begin{equation}
\label{esti-KF}
\begin{aligned}
&\mathrm{ess}\hspace{-0.2cm}\sup_{0\le t\le T} \left(\|\partial_t K^j F(t)\|_{L^2(\Omega)}^2 
+ \|K^j F(t)\|_{H^2(\Omega)}^2 \right) \\
\le&\; C\left( \|u_0\|_{H^2(\Omega)}^2 + \|f\|_{H^1(0,T;L^2(\Omega))}^2 \right),\quad j=0,1,\ldots,N-1.
\end{aligned}
\end{equation}
In the end, collecting the above estimates \eqref{eq-Ku4}, \eqref{esti-KF}, 
and recalling \eqref{esti-sec}, we conclude that 
$u\in W^{1,\infty}(0,T;L^2(\Omega))\cap L^\infty(0,T;H^2(\Omega)\cap H_0^1(\Omega))$ 
with the estimate:
\begin{align*}
\mathrm{ess}\hspace{-0.2cm}\sup_{0\le t\le T} \left(\|\partial_t u(t)\|_{L^2(\Omega)}^2 
+ \|u(t)\|_{H^2(\Omega)}^2 \right) 
\le C\left( \|u_0\|_{H^2(\Omega)}^2 + \|f\|_{H^1(0,T;L^2(\Omega))}^2 \right).
\end{align*}
We finish the last part of Theorem \ref{thm-uniq-exist}.

\section{Long-time asymptotics}
\label{sec-asymp-long}

In this section, we establish the long-time asymptotic estimate for the solution $u$ to 
the initial-boundary value problem \eqref{equ-u}--\eqref{equ-ibc}. 
The proof relies on a suitable energy estimate and the use of the asymptotic behavior for 
a related ordinary fractional differential equation. 

To start with, some important auxiliary results as follows are established. 
Henceforth, $(\cdot ,\cdot)$ denotes the scalar product in $L^2(\Omega)$. 
We have the following coercivity inequality for the Caputo derivative.

\begin{lem}
\label{lem1}
Let $y\in H^1(0,T;L^2(\Omega))$. Then 
$$
(y(t), \partial_t^\alpha y(t)) \ge \|y(t)\|_{L^2(\Omega)} \, \partial_t^\alpha \|y(t)\|_{L^2(\Omega)}
$$
holds true for {$0 < t < T$}. 
\end{lem}
\begin{proof}
The proof is done by direct calculations. For simplicity, we set 
$g_\alpha (t) := \frac{t^{-\alpha}}{\Gamma(1-\alpha)}$ and its derivative 
$g_\alpha^\prime (t) = -\frac{\alpha t^{-\alpha-1}}{\Gamma(1-\alpha)}$, $0 < t \le T$, 
and we denote
$$
I(t) := (y(t), \partial_t^\alpha y(t)) - \|y(t)\|_{L^2(\Omega)} \, \partial_t^\alpha \|y(t)\|_{L^2(\Omega)}.
$$
Then it is sufficient to prove $I \ge 0$. 
For this, we divide $I(t)$ into two parts: $I(t) = I_1(t) + I_2(t)$ with
\begin{align*}
&I_1(t) := (y(t), \partial_t^\alpha y(t)) - \frac{1}{2} \partial_t^\alpha \|y(t)\|_{L^2(\Omega)}^2,\\
&I_2(t) := \frac{1}{2} \partial_t^\alpha \|y(t)\|_{L^2(\Omega)}^2 
- \|y(t)\|_{L^2(\Omega)} \, \partial_t^\alpha \|y(t)\|_{L^2(\Omega)}.
\end{align*}

By Fubini's theorem and the definition of Caputo fractional derivative, we find
\begin{align*}
I_1(t) 
&= \int_0^t g_\alpha (t-\tau) (y(t), \partial_\tau y(\tau)) d\tau - \int_0^t g_\alpha(t-\tau) 
(y(\tau),\partial_\tau y(\tau))d\tau\\
&= \int_0^t g_\alpha (t-\tau) (y(t)-y(\tau), \partial_\tau y(\tau)) d\tau\\
&= - \frac{1}{2} \int_0^t g_\alpha(t-\tau) \partial_\tau \|y(t)-y(\tau)\|_{L^2(\Omega)}^2 d\tau. 
\end{align*}
Then integration by parts yields
\begin{align*}
I_1(t) =
& -\frac{1}{2} g_\alpha(t-\tau)\|y(t)-y(\tau)\|_{L^2(\Omega)}^2 \Big|_{\tau=0}^{\tau=t} \\
& - \frac{1}{2} \int_0^t g_\alpha^\prime (t-\tau) \|y(t)-y(\tau)\|_{L^2(\Omega)}^2 d\tau.
\end{align*}
Moreover, we claim that 
\begin{equation}
\label{claim-g}
\lim_{\tau\to t}g_\alpha(t-\tau)\|y(t)-y(\tau)\|_{L^2(\Omega)}^2 =0.
\end{equation}
Indeed, by noting that
\begin{align*}
&g_\alpha(t-\tau)\|y(t)-y(\tau)\|_{L^2(\Omega)}^2\\
\le&\; C (t-\tau)^{-\alpha} \left\| \int_\tau^t|\partial_s y(s)| ds \right\|_{L^2(\Omega)}^2\\
\le&\; C (t-\tau)^{-\alpha} \int_\tau^t \|\partial_s y(s)\|_{L^2(\Omega)}^2 ds  \int_\tau^t 1^2 ds
\le C (t-\tau)^{1-\alpha} \|y\|_{H^1(0,T;L^2(\Omega))}^2,
\end{align*}
where in the last line we used H\"older's inequality and Fubini's theorem. 
Thus the claim \eqref{claim-g} is true and we see that
\begin{align*}
I_1(t) 
= \frac{1}{2} g_\alpha(t)\|y(t)-y(0)\|_{L^2(\Omega)}^2 - \frac{1}{2} \int_0^t g_\alpha^\prime(t-\tau) 
\|y(t)-y(\tau)\|_{L^2(\Omega)}^2 d\tau.
\end{align*}

For $I_2$, we note that the triangle inequality 
$\|y(t)\|_{L^2(\Omega)} - \|y(\tau)\|_{L^2(\Omega)} \le \|y(t)-y(\tau)\|_{L^2(\Omega)}$, 
so that 
$\lim_{\tau\to t} g_\alpha(t-\tau)\left(\|y(t)\|_{L^2(\Omega)} - \|y(\tau)\|_{L^2(\Omega)}\right)^2 = 0$, 
then by an argument similar to the calculation for $I_1$, we find
\begin{align*}
I_2(t) 
&= \int_0^t g_\alpha(t-\tau) \left(\|y(\tau)\|_{L^2(\Omega)} - \|y(t)\|_{L^2(\Omega)}\right) \partial_\tau 
\|y(\tau)\|_{L^2(\Omega)} d\tau\\
&= \frac{1}{2} \int_0^t g_\alpha(t-\tau) \partial_\tau 
\left(\|y(t)\|_{L^2(\Omega)}-\|y(\tau)\|_{L^2(\Omega)}\right)^2 d\tau\\
&= - \frac{1}{2} g_\alpha(t)\left(\|y(t)\|_{L^2(\Omega)}-\|y(0)\|_{L^2(\Omega)}\right)^2 \\
&\quad+ \frac{1}{2} \int_0^t g_\alpha^\prime(t-\tau) 
\left(\|y(t)\|_{L^2(\Omega)}-\|y(\tau)\|_{L^2(\Omega)}\right)^2 d\tau.
\end{align*}
Therefore, by noting 
$$
\|y(t)-y(\tau)\|_{L^2(\Omega)}^2 = \|y(t)\|_{L^2(\Omega)}^2 + \|y(\tau)\|_{L^2(\Omega)}^2 - 2(y(t),y(\tau))
$$
for $0\le \tau \le t$, we obtain
\begin{align*}
I(t) 
&= g_\alpha(t)\left(\|y(t)\|_{L^2(\Omega)}\|y(0)\|_{L^2(\Omega)} - (y(t),y(0))\right) \\
&\quad - \int_0^t g_\alpha^\prime(t-\tau) 
\left(\|y(t)\|_{L^2(\Omega)}\|y(\tau)\|_{L^2(\Omega)} - (y(t),y(\tau))\right) d\tau.
\end{align*}
Finally, H\"older's inequality and $g_\alpha > 0$, $g_\alpha^\prime < 0$ in $(0,T)$ 
imply $I(t) \ge 0$ for $0< t < T$, which completes the proof of the lemma.
\end{proof}

\begin{lem}
\label{lem-mp}
Let $\lambda >0$ and $p_0 \ge 0$ be constants. Assume that $w\in H^1(0,T)$ satisfies
\begin{equation}\label{equ-lem-mp}
\begin{cases}
\!\begin{alignedat}{2}
& \partial_t w(t) + p_0\partial_t^\alpha w(t) + \lambda w(t) \leq 0, \quad 0 < t < T,\\
& w(0) \le 0.
\end{alignedat}
\end{cases}
\end{equation}
Then $w(t) \le 0$ for $0<t\le T$. 
\end{lem}
\begin{proof}
We start the proof in the case of $w\in C^1[0,T]$. 
By continuity, we find \eqref{equ-lem-mp} holds true for $t\in [0,T]$. 
We prove the lemma by contradiction.
Assume that $w$ is positive at some point in $(0,T]$. 
Then $w$ attains its positive maximum in $(0,T]$, that is, there exists $t_0\in (0,T]$ such that 
$w(t_0)>0$ and $w(t_0)\ge w(t)$ for $t\in [0,T]$. 
Immediately, we have $\partial_t w(t_0) \ge 0$. 
With reference to \cite[Theorem 1]{Luchko09b}, 
we find $\partial_t^\alpha w(t_0) \ge 0$. 
Thus, we obtain
$$
\partial_t w(t_0) + p_0\partial_t^\alpha w(t_0) + \lambda w(t_0) > 0,
$$
which is a contradiction to \eqref{equ-lem-mp}. 

Next, we assume $w\in H^1(0,T)$. 
For any nonnegative function $\varphi\in C^1[0,T]$, we denote 
$\overline w:=w*\varphi:=\int_0^t w(t-\tau)\varphi(\tau)d\tau$. 
It is not difficult to see that $\overline w\in C^1[0,T]$ and satisfies 
$$
\begin{cases}
\begin{alignedat}{2}
& \partial_t \overline w(t) + p_0\partial_t^\alpha \overline w(t) + \lambda \overline w(t) \leq 0, 
\quad 0< t\le T,\\
& \overline w(0) = 0.
\end{alignedat}
\end{cases}
$$
Therefore, from the above argument, it follows that $\overline w(t)\le 0$ for any $t\in (0,T]$, that is,
$$
\int_0^t w(t-\tau)\varphi(\tau)d\tau \le 0,\quad 0< t\le T.
$$
Since $\varphi\in C^1[0,T]$ is nonnegative and can be arbitrarily chosen, 
we must have $w\le 0$ in $(0,T]$. 
Indeed, noting that $w\in H^1(0,T) \supset C[0,T]$, if $w\le 0$ fails in $(0,T]$, 
then we can choose $t_0\in(0,T)$ and sufficiently small constant $\varepsilon>0$ 
such that $w(t)>0$ for any $t\in [t_0-\varepsilon,t_0+\varepsilon]$. 
We then construct $\varphi\in C^1[0,T]$ satisfying 
$$
\varphi(t)=
\begin{cases}
1 & \mbox{ if } t\in [t_0-\frac{\varepsilon}{2},t_0+\frac{\varepsilon}{2}]\\
0 & \mbox{ if } t\in (0,t_0-\varepsilon)\cup(t_0+\varepsilon,T).
\end{cases}
$$
In this case, we calculate the convolution $w*\varphi$ and find that
$$
\int_0^t w(t-\tau)\varphi(\tau)d\tau 
\ge \int_{t_0-\frac{\varepsilon}{2}}^{t_0+\frac{\varepsilon}{2}} w(t-\tau) d\tau
\ge \varepsilon\inf_{(t_0-\frac{\varepsilon}{2},t_0+\frac{\varepsilon}{2})} w(t) >0,
$$
which is a contradiction. 
We must have $w(t)\le 0$ for any $t\in (0,T]$. This completes the proof of the lemma.
\end{proof}

\begin{lem}
\label{lem-alpha}
Let $\lambda >0$, $p\in L^\infty(0,T)$ 
and $p_0 \le p(t) \le p_1$, $t\in (0,T)$ for some positive constants $p_0,p_1>0$. 
Assume that $z\in W^{1,\infty}(0,T)$ satisfies $z(t)\ge 0$ and 
\begin{equation}\label{equ-lem-mp1}
\partial_t z(t) + p(t)\partial_t^\alpha z(t) + \lambda z(t) \leq 0, \quad \mbox{} 0< t< T,
\end{equation}
Then $\partial_t^\alpha z(t) \le 0$ for $0< t< T$. 
\end{lem}
\begin{proof}
Step 1. 
We first assume $z\in C^2[0,T]$ and we find $\partial_t^\alpha z(0) = 0$,  
which can be easily verified by the following estimate:
$$
|\partial_t^\alpha z(t)| 
\le \int_0^t \frac{(t-s)^{-\alpha}}{\Gamma(1-\alpha)} |\partial_s z(s)| ds
\le \frac{\|z\|_{C^1[0,T]}}{\Gamma(2-\alpha)} t^{1-\alpha}, \quad t\in [0,T].
$$

Now we set $\widetilde z := \partial_t^\alpha z$, then $J^\alpha \widetilde z = z - z(0)$, 
where $J^\alpha$ denotes the Riemann-Liouville integral operator. 
By \eqref{equ-lem-mp1}, we obtain
\begin{equation}\label{eq7}
\begin{cases}
\begin{alignedat}{2}
& \partial_t J^\alpha \widetilde z(t) + p(t)\widetilde z(t) \le -\lambda  z(t) \leq 0, 
\quad & 0< t< T,\\
& \widetilde z(0) = \partial_t^\alpha z(0) = 0.
\end{alignedat}
\end{cases}
\end{equation}
Then we claim that $\widetilde z(t)\le 0$ for any $t\in (0,T]$. 
Otherwise, there exists $t_1\in (0,T]$ such that $\widetilde z$ attains its positive maximum 
$\widetilde z(t_1)>0$ at point $t_1$. 
By Theorem~{\bf 2.1} in \cite{AL14}, we have the Riemann-Liouville fractional derivative at 
$t=t_1$ satisfies 
$$
\partial_t J^\alpha \widetilde z(t_1) \ge \frac{t_1^{\alpha-1}}{\Gamma(\alpha)}\widetilde z(t_1) > 0,
$$ 
and hence $\partial_t J^\alpha \widetilde z > 0$ in a neighborhood of $t_1$, 
which yields a contraction to \eqref{eq7}. 
Thus, $\widetilde z$ is non-positive and then we have $\partial_t^\alpha z(t) \le 0$ for $0<t\le T$. 

Step 2. We assume $z\in W^{1,\infty}(0,T)$. We denote 
$z_\mu(t) := z(t)+\mu^{-1}E_{\alpha,1}(-\mu t^\alpha)$ with $\mu>0$. 
Then we see that $z_\mu$ is positive and satisfies the following equation
\begin{equation}
\label{eq-zmu}
(\partial_t  + p(t)\partial_t^\alpha  + \lambda) z_\mu 
= (\partial_t + p(t) \partial_t^\alpha  +\lambda) z + R_\mu(t)
\end{equation}
where 
$R_\mu(t):=-t^{\alpha-1} E_{\alpha,\alpha}(-\mu t^\alpha) - p(t) E_{\alpha,1}(-\mu t^\alpha) 
+ \frac{\lambda}{\mu} E_{\alpha,1}(-\mu t^\alpha)$ 
can be easily derived by differential properties \eqref{eq-ml1} and \eqref{eq-ml2} of 
the Mittag-Leffler function. 
Moreover, by the useful estimate \eqref{esti-ML} for the Mittag-Leffler functions, 
we can see that there exists a constant $\delta_\mu>0$ such that the following inequality 
$$
- t^{\alpha-1} E_{\alpha,\alpha}(-\mu t^\alpha) - p(t) E_{\alpha,1}(-\mu t^\alpha) 
+ \frac{\lambda}{\mu} E_{\alpha,1}(-\mu t^\alpha) < -\delta_\mu
$$
is valid for any $t\in(0,T]$ and $\mu>>1$, 
which combined with the equality \eqref{eq-zmu} and the assumption \eqref{equ-lem-mp1} implies
\begin{align*}
(\partial_t  + p(t)\partial_t^\alpha  + \lambda) z_\mu 
\le -\delta_\mu,\quad \mbox{}0<t<T.
\end{align*}
Then for any $\varepsilon>0$, we can choose $z_{\mu,\varepsilon}\in C^2[0,T]$ 
such that $z_{\mu,\varepsilon}\ge 0$ and 
$$
\|z_{\mu,\varepsilon} - z_\mu\|_{W^{1,\infty}(0,T)} \le \varepsilon.
$$
By a direct calculation, we see that 
$$
\begin{aligned}
(\partial_t  + p(t)\partial_t^\alpha  + \lambda) z_{\mu,\varepsilon}
= &\; (\partial_t  + p(t)\partial_t^\alpha  + \lambda) z_{\mu} +(\partial_t + p(t) \partial_t^\alpha  
+\lambda) (z_{\mu,\varepsilon} - z_\mu)\\
\le &\; -\delta_\mu + \left(1+ \frac{p_1 T^{1-\alpha}}{\Gamma(2-\alpha)} + \lambda \right)\varepsilon.
\end{aligned}
$$
Consequently, letting $\varepsilon <<1$, we see that
\begin{align*}
(\partial_t  + p(t)\partial_t^\alpha + \lambda) z_{\mu,\varepsilon} \le 0.
\end{align*}
Now by step 1, it follows that 
$\partial_t^\alpha z_{\mu,\varepsilon}\le 0$ for any sufficiently small $\varepsilon>0$. 
Letting $\varepsilon\to 0$ and we have 
$\partial_t^\alpha z_\mu\le 0$ for any sufficiently large $\mu$. 
Finally, again from the estimate \eqref{esti-ML} for Mittag-Leffler functions, 
we see that $E_{\alpha,1}(-\mu t^\alpha)$ tends to $0$ as $\mu\to \infty$, 
hence that $\partial_t^\alpha z\le 0$ by letting $\mu\to\infty$. 
We then finish the proof of the lemma.
\end{proof}

Equipped with the above lemmata, we prove our main results by applying an energy estimate. 

\begin{proof}[Proof of Theorem \ref{thm-asymp-long}]

According to the result of the forward problem (Theorem \ref{thm-uniq-exist}), we note that 
$u\in H^1(0,T;H^{-1}(\Omega)) \cap L^2(0,T;H_0^1(\Omega)) \cap C([0,T];L^2(\Omega))$
provided that $u_0\in L^2(\Omega)$. 
Since this regularity is not enough to guarantee the above lemmata that we will use in the proof, 
we need introduce the approximate solutions $\{u_N\}_{N=1}^\infty$ which solve
\begin{equation}
\label{equ-uN}
\left\{
\begin{aligned}
& \partial_t u_N + q(t) \partial_t^\alpha u_N = - A u_N + c(x,t) u_N, 
&&\quad (x,t)\in \Omega \times (0,T),\\
& u_N(x,t) = 0, 
&&\quad (x,t)\in \partial\Omega \times (0,T),\\
& u_N(x, 0) = \sum_{k=1}^N (u_0, \varphi_k)_{L^2(\Omega)}\varphi_k, 
&&\quad x\in \Omega
\end{aligned}
\right.
\end{equation}
for each $N\in \mathbb{N}$. 
Here we recall that $\{\varphi_k\}_{k=1}^\infty \subset H^2(\Omega)\cap H_0^1(\Omega)$ is 
the set of the eigenfunctions of $A$ with homogeneous Dirichlet boundary condition and forms 
an orthonormal basis of $L^2(\Omega)$. 
By the third part of Theorem \ref{thm-uniq-exist}, we see that 
$u_N \in W^{1,\infty}(0,T;L^2(\Omega))$, which guarantees the regularity when we apply 
Lemmata {\bf 4.1}--{\bf 4.3} in the following context. 

Now we multiply $u_N$ on both sides of the first equation of \eqref{equ-uN} and 
integrate over $\Omega$. 
Integration by parts yields
\begin{equation}\label{eq1}
(\partial_t u_N, u_N) + q(t)(\partial_t^\alpha u_N, u_N) 
+ \sum_{i,j=1}^d (a_{ij}\partial_{x_i} u_N,\partial_{x_j} u_N) - (c(t)u_N, u_N) = 0
\end{equation}
for $0<t< T$. 
Next we estimate the left-hand side of \eqref{eq1} from below. 

The ellipticity of the operator $A$ and the Poincar\'e inequality imply
\begin{equation}\label{eq2}
\sum_{i,j=1}^d (a_{ij}\partial_{x_i} u_N(t),\partial_{x_j} u_N(t)) 
\ge \nu \|\nabla u_N(t)\|_{L^2(\Omega)}^2 \ge \lambda \|u_N(t)\|_{L^2(\Omega)}^2
\end{equation}
with some positive constant $\lambda>0$, which depends only on $\nu$ and $\Omega$. 
By Lemma~{\bf 4.1}, we have
\begin{equation}\label{eq3}
(\partial_t^\alpha u_N(t), u_N(t)) \ge 
\|u_N(t)\|_{L^2(\Omega)} \, \partial_t^\alpha \|u_N(t)\|_{L^2(\Omega)}
\end{equation} 
for $0<t< T$. 
Since $c \le 0$ in $\Omega \times (0,T)$ and 
\begin{equation}\label{eq4}
(\partial_t u_N(t), u_N(t)) = \frac{1}{2}\partial_t \|u_N(t)\|_{L^2(\Omega)}^2 
= \|u_N(t)\|_{L^2(\Omega)}\, \partial_t \|u_N(t)\|_{L^2(\Omega)},
\end{equation}
we insert \eqref{eq2}--\eqref{eq4} into \eqref{eq1} and obtain
\begin{equation}\label{eq5}
\|u_N(t)\|_{L^2(\Omega)} \left(\partial_t \|u_N(t)\|_{L^2(\Omega)} 
+ q(t)\partial_t^\alpha \|u_N(t)\|_{L^2(\Omega)} + \lambda\|u_N(t)\|_{L^2(\Omega)} \right) \leq 0
\end{equation}
for $0<t< T$. 
We assert that 
\begin{equation}\label{eq6}
\partial_t \|u_N(t)\|_{L^2(\Omega)} + q(t)\partial_t^\alpha \|u_N(t)\|_{L^2(\Omega)} 
+ \lambda\|u_N(t)\|_{L^2(\Omega)} \leq 0
\end{equation}
for $0<t< T$. 
If \eqref{eq6} does not hold, then there exists $t_2\in (0,T]$ such that 
\begin{equation*}
\partial_t \|u_N(t_2)\|_{L^2(\Omega)} + q(t_2)\partial_t^\alpha \|u_N(t_2)\|_{L^2(\Omega)} 
+ \lambda\|u_N(t_2)\|_{L^2(\Omega)} > 0. 
\end{equation*}
Then by \eqref{eq5}, we see that $\|u_N(t_2)\|_{L^2(\Omega)} = 0$, which indicates that 
$\|u_N(t)\|_{L^2(\Omega)}$ attains its minimum at $t=t_2$. 
Immediately we have $\partial_t \|u_N(t_2)\|_{L^2(\Omega)} \le 0$, 
and $\partial_t^\alpha \|u_N(t_2)\|_{L^2(\Omega)} \le 0$ from Theorem~{\bf 1} 
in \cite{Luchko09b}. 
This yields a contradiction. 

Next we estimate $\|u_N(t)\|_{L^2(\Omega)}$ by some function from above. 
We introduce an auxiliary function $v$ which solves the following ordinary fractional differential 
equation: 
\begin{equation}\label{equ-ode}
\begin{cases}
\begin{alignedat}{2}
& \partial_t v(t) + q_1\partial_t^\alpha v(t) + \lambda v(t) = 0, 
\quad & t>0,\\
& v(0) = \|u_0\|_{L^2(\Omega)}.
\end{alignedat}
\end{cases}
\end{equation}
Here we recall that $q_1$ is a positive constant and $q(t)\le q_1$ for $t>0$. 
Let $w_N(t) = \|u_N(t)\|_{L^2(\Omega)} - v(t)$. 
Since $T>0$ is arbitrary, by \eqref{equ-ibc}, \eqref{eq6} and \eqref{equ-ode} and 
$$
\|u_N(0)\|_{L^2(\Omega)} = \left(\sum_{k=1}^N (u_0,\varphi_k)_{L^2(\Omega)}^2 \right)^{\frac12}
\le \|u_0\|_{L^2(\Omega)} = v(0),
$$
we obtain
\begin{equation}\label{equ-ode2}
\begin{cases}
\begin{alignedat}{2}
& \partial_t w_N + q_1\partial_t^\alpha w_N + \lambda w_N 
\leq (q_1-q(t))\partial_t^\alpha \|u_N(t)\|_{L^2(\Omega)}, 
\quad & 0<t<T,\\
& w_N(0) \le 0.
\end{alignedat}
\end{cases}
\end{equation}
From \eqref{eq6}, applying Lemma~{\bf 4.3}, we can see that 
$\partial_t^\alpha \|u_N(t)\|_{L^2(\Omega)} \le 0$, 
which means the right-hand side of \eqref{equ-ode2} is not positive. 
Now we can apply Lemma~{\bf 4.2} for \eqref{equ-ode2} to obtain $w_N(t)\le 0$ for $0<t<T$, that is,
\begin{equation*}
\|u_N(t)\|_{L^2(\Omega)} \le v(t)\quad \mbox{ for }0<t<T.
\end{equation*}
By Theorem \ref{thm-uniq-exist}, we find that for arbitrarily fixed $T>0$, $u_N(t)$ 
converges to $u(t)$ in $L^2(\Omega)$ for any $0<t<T$. 
Moreover, we note that $v$ is the solution to \eqref{equ-ode}, which is independent of $N$. 
Thus, we have
\begin{equation*}
\|u(t)\|_{L^2(\Omega)} \le v(t)\quad \mbox{ for }0<t<T.
\end{equation*}
Finally, since $T>0$ can be arbitrarily fixed and $v$ is also independent of $T$, 
it remains to discuss the long-time asymptotic behavior of $v$.

By applying the Laplace transform to the ordinary fractional diffusion equation 
\eqref{equ-ode}, we can derive 
$$
|v(t)| \le C\|u_0\|_{L^2(\Omega)} t^{-\alpha}, \quad t\ge t_0
$$
for arbitrarily fixed $t_0>0$. We put the details in Lemma {\bf A.1} in the appendix. 
This completes the proof of Theorem~\ref{thm-asymp-long}. 
\end{proof}

\section{Conclusions and open problems}
In this paper, we considered the diffusion equation with fractional derivative on the bounded 
multi-dimensional domain subject to a homogeneous Dirichlet boundary condition. 
Firstly, by regarding the fractional term as a source, we transferred the differential equation to 
an equivalent integral form, 
then we used the Fredholm alternative for the compact operator to show the well-posedness for 
the forward problem, which is essential for numerically analyzing this type of problems and for 
dealing with the inverse problems for the fractional diffusion equation. 
On the basis of the forward problem, the energy estimate and maximum principle allow us to 
obtain the asymptotic decay in time for the solution to the initial-boundary value problem 
\eqref{equ-u}--\eqref{equ-ibc}.

For the sake of simplicity, we consider the case of only one fractional derivative 
in this paper. As one can see from the proof, we can similarly prove Theorem \ref{thm-uniq-exist}  
for a multi-term time-fractional diffusion equation: 
$$
\partial_t u + \sum_{j=1}^\ell q_j(t)\partial_t^{\alpha_j} u 
= - Au + c(x,t)u + f(x,t), \quad (x,t)\in \Omega\times (0,T), 
$$ 
where $\ell\in \mathbb{N}$ is given and we assume $0<\alpha_1<\alpha_2<\ldots<\alpha_\ell<1$. 
Moreover, if we assume further $f=0$, $c\le 0$ and $q_j(t)=q_j\ge 0$, $j=2,\ldots,\ell$, 
$p_0\le q_1(t)\le p_1$ with some positive constants $p_1\ge p_0>0$, 
then by following the proof in Section {\bf 4} and Appendix, Theorem \ref{thm-asymp-long} can be 
immediately generalized in the multi-term case with the following long-time asymptotic estimate
$$
\|u(\,\cdot\,,t)\|_{L^2(\Omega)}
\leq C\|u_0\|_{L^2(\Omega)}t^{-\alpha_1}, 
$$
which indicates that the asymptotic behavior of the solution depends on the lowest order of 
the fractional derivatives. 
The assumption that $q_j$, $j=2,\ldots,\ell$ are nonnegative constants may be relaxed by modifying 
the argument we used in this paper but here we do not discuss more details.

As for the open problems related to the initial-boundary value problems for the fractional diffusion 
equations, let us mention the following ones: 
In the proofs of our results, we needed the assumption that $q$ is independent of $x$, 
which is necessary for deriving Lemma {\bf 4.1}. 
It would be interesting to investigate what happens with the asymptotic properties of the solution 
if this assumption is relaxed. 

Other interesting directions of the research would be that 
whether the estimate is valid for the fractional diffusion with nonlinearity. 
It still remains open and should be investigated. 

\section*{Acknowledgments}
The first author thanks National Natural Science Foundation of China 11801326. 
The second author was supported by Japan Society for the Promotion of Science under the program of JSPS Postdoctoral Fellowships for Research in Japan.
The third authors was supported by Grant-in-Aid for Scientific Research (A) 20H00117 
of Japan Society for the Promotion of Science, 
The National Natural Science Foundation of China (No. 11771270, 91730303), 
and the RUDN University Strategic Academic Leadership Program. 
This work was also supported by A3 Foresight Program \lq Modeling and Computation of Applied 
Inverse Problems' of Japan Society for the Promotion of Science and the Research Institute for 
Mathematical Sciences, an International Joint Usage/Research Center located in Kyoto University.

\appendix

\section{Appendix}
In this part, we will follow the argument used in \cite[Section 4]{GM08} to give the proof for 
the long-time asymptotic behavior of the solution $v$ to the following fractional ordinary equation
\begin{equation}\label{eq-ode}
\begin{cases}
\begin{alignedat}{2}
& \partial_t v(t) + q_1\partial_t^\alpha v(t) + \lambda v(t) = 0, 
\quad & t>0,\\
& v(0) = v_0.
\end{alignedat}
\end{cases}
\end{equation}
\begin{lem}
\label{lem-asy}
Assume $q_1>0$, $\lambda>0$ and $v_0\ne0$ are given constant. Then the solution $v$ to 
the problem \eqref{eq-ode} admits the following long-time asymptotic estimate
$$
|v(t)| \le C|v_0| t^{-\alpha}, \quad t \ge t_0
$$ 
for any $t_0>0$. 
Here the order $\alpha$ is sharp and the constant $C$ depends only on $q_1,\alpha,\lambda$ and $t_0$.
\end{lem}
\begin{proof}
By applying the Laplace transform to the ordinary fractional diffusion equation \eqref{eq-ode}, 
we find that
$$
\mathcal{L}[v](s) = \frac{1 + q_1 s^{\alpha-1}}{s + q_1 s^\alpha + \lambda} v_0,
\quad s>0
$$
where $\mathcal{L}[v]$ denotes the Laplace transform of the function $v$. 
Then we get $v(t)$ by the Fourier-Mellin transform of $\mathcal{L}[v](s)$. 
Since it is readily to see that $s + q_1 s^\alpha + \lambda$ has no zero in the main sheet of 
the Riemann surface including the negative real axis, 
we can deform the original Bromwich path into the Hankel path $Ha(\varepsilon)$ and obtain
\begin{equation*}
v(t) 
= \frac{1}{2\pi i} v_0 \int_{Ha(\varepsilon)} e^{st}\frac{1 + q_1 s^{\alpha-1}}{s + q_1 s^\alpha + \lambda} 
ds.
\end{equation*}
Here the Hankel path $Ha(\varepsilon)$ is the loop which starts from $-\infty$ along the lower side of 
the negative real axis, encircles the circular disc $|s|=\varepsilon$ and ends at $-\infty$ along 
the upper side of the negative real axis. 
Letting $\varepsilon \to 0$ yields
\begin{equation*}
v(t) 
= v_0 \int_0^\infty e^{-rt} H_{\alpha,0}^{(1)}(r;q_1,\lambda) dr,
\end{equation*}
with
\begin{align*}
H_{\alpha,0}^{(1)}(r;q_1,\lambda) 
&= -\frac{1}{\pi} \Im 
\left\{\frac{1 + q_1 s^{\alpha-1}}{s + q_1 s^\alpha + \lambda}\bigg|_{s=re^{i\pi}}\right\} \\
&= \frac{1}{\pi} \frac{\lambda q_1 r^{\alpha-1} \sin{(\alpha \pi)}}{(\lambda-r)^2 + q_1^2 r^{2\alpha} 
+ 2(\lambda-r)q_1 r^\alpha \cos{(\alpha \pi)}}
\end{align*}
where $\Im z$ denotes the imaginary part of $z\in \mathbb{C}$. 
We break the above integral into two parts as follows
\begin{equation*}
v(t) 
= v_0 \int_0^\delta e^{-rt} H_{\alpha,0}^{(1)}(r;q_1,\lambda) dr
+ v_0 \int_\delta^\infty e^{-rt} H_{\alpha,0}^{(1)}(r;q_1,\lambda) dr
=: I_1+I_2,
\end{equation*}
where $0<\delta\le \lambda$ will be chosen later. We will estimate $I_1$ and $I_2$ separately. 
For $I_1$, in view of the inequality that
\begin{align*}
&(\lambda-r)^2 + q_1^2 r^{2\alpha} + 2(\lambda-r)q_1 r^\alpha \cos{(\alpha \pi)}
\\
\ge& (\lambda-r)^2 + q_1^2 r^{2\alpha} - 2(\lambda-r)q_1 r^\alpha
=  (\lambda - r - q_1r^\alpha)^2, \quad 0<r<\delta,
\end{align*}
we can choose $\delta>0$ being sufficiently small such that 
$\lambda - r - q_1r^\alpha\ge \frac\lambda2$ for any $0<r<\delta$. 
Consequently, we arrive at the following inequalities
\begin{align*}
|I_1(t)| \le& \frac{2|v_0|q_1 \sin(\alpha\pi)}\pi \int_0^\delta e^{-rt} r^{\alpha-1} dr
\\
\le& \frac2\pi q_1 \sin(\alpha\pi) \Gamma(\alpha) |v_0| t^{-\alpha},\quad t>0.
\end{align*}
Next, we estimate $I_2$. Firstly, for any $s=re^{i\pi}$ with $r>0$, a direct calculation yields 
$$
|s+q_1s^\alpha + \lambda| \ge \Im s + q_1 \Im s^\alpha
= |s| \sin\pi + q_1 |s|^\alpha \sin (\alpha\pi)
= q_1\sin (\alpha\pi) r^\alpha >0.
$$
Hence we see that 
$$
|H_{\alpha,0}^{(1)}(r;q_1,\lambda)|
\le \frac{1}{\pi} \left| \frac{1 + q_1 s^{\alpha-1}}{s + q_1 s^\alpha + \lambda}\bigg|_{s=re^{i\pi}}\right|
\le \frac{ 1 + q_1 r^{\alpha-1}}{\pi q_1 \sin (\alpha\pi) r^\alpha}
\le \frac{ 1 + q_1 r^{\alpha-1}}{\pi q_1 \sin (\alpha\pi) \delta^\alpha}
$$
holds true for any $r\ge \delta$. Therefore, we obtain 
\begin{align*}
|I_2(t)| 
&\le \frac{|v_0|}{\pi q_1 \sin (\alpha\pi) \delta^\alpha}\int_\delta^\infty e^{-rt} (1 + q_1 r^{\alpha-1}) dr
\\
&\le \frac{|v_0|}{\pi q_1 \sin (\alpha\pi) \delta^\alpha} 
\left(\frac1t + q_1\Gamma(\alpha) t^{-\alpha}\right),\quad t>0.
\end{align*}
Finally, collecting all the above estimates for $I_1$ and $I_2$, we arrive at the inequality
$$
|v(t)| \le C|v_0| (t^{-1}+t^{-\alpha}), \quad t>0,
$$
and thus, by noting $t^{-1} = t^{\alpha-1} t^{-\alpha} \le t_0^{\alpha-1} t^{-\alpha}$ for $t\ge t_0$, 
we have
$$
|v(t)| \le C|v_0| t^{-\alpha}, \quad t\ge t_0
$$
where the constant $C>0$ depends only on $q_1$, $\lambda$, $\alpha$ and $t_0$. 
Moreover, for any $r>0$, we have $H_{\alpha,0}^{(1)}(r;q_1,\lambda)>0$ and 
for any $0\le r\le 1$, we have the inequality
$$
(\lambda-r)^2+q_1^2r^{2\alpha} + 2(\lambda-r)q_1 r^\alpha \cos (\alpha\pi) 
\le (|\lambda-r|+q_1 r^\alpha)^2 \le (\lambda+q_1)^2,     
$$
which implies
\begin{align*}
|v(t)|
&\ge |v_0| \int_0^1 e^{-rt} H_{\alpha,0}^{(1)}(r;q_1,\lambda) dr\\
&\ge |v_0| \frac{\lambda q_1 \sin (\alpha\pi)}{\pi (\lambda+q_1)^2}\int_0^1 e^{-rt} r^{\alpha-1} dr\\
&= |v_0| t^{-\alpha} \frac{\lambda q_1 \sin (\alpha\pi)}{\pi (\lambda+q_1)^2}
\int_0^t e^{-r} r^{\alpha-1} dr\\
&\ge |v_0| t^{-\alpha} \frac{\lambda q_1 \sin (\alpha\pi)}{\pi (\lambda+q_1)^2}
\int_0^{t_0} e^{-r} r^{\alpha-1} dr, \quad t\ge t_0.
\end{align*}
Thus, we find that the decay rate $t^{-\alpha}$ is sharp and we finish the proof of the lemma.
\end{proof}



\begin{thebibliography}{99}

\bibitem{AL14}
M. Al-Refai, Y. Luchko,
Maximum principle for the fractional diffusion equations with the Riemann-Liouville fractional derivative and its applications.
\emph{Fract. Calc. Appl. Anal.} \textbf{17} (2014), 483--498. 

\bibitem{C95}
M. Caputo, 
Mean fractional-order-derivatives differential equations and filters. 
\emph{Annali dellUniversit\`adi Ferrara} \textbf{41} (1995), 73--84.

\bibitem{C99}
M. Caputo, 
Diffusion of fluids in porous media with memory. 
\emph{Geothermics} \textbf{28} (1999), 113--130.

\bibitem{CLY17}
X. Cheng, Z. Li, M. Yamamoto,  
Asymptotic behavior of solutions to space-time fractional diffusion equations. 
\emph{Mathematical Methods in the Applied Sciences} \textbf{40} (2017), 1019--1031.

\bibitem{E98}
L. Evans, 
\emph{Partial Differential Equations}.
American Mathematical Society, 1998.


\bibitem{GLY15}
R. Gorenflo, Y. Luchko, M. Yamamoto,
Time-fractional diffusion equation in the fractional Sobolev spaces. 
\emph{Fract. Calc. Appl. Anal.} \textbf{18} (2015), 799--820. 
doi: https://doi.org/10.1515/fca-2015-0048

\bibitem{GM08}
R. Gorenflo, F. Mainardi, 
Fractional Calculus: Integral and Differential Equations of Fractional Order. 
In: A. Carpinteri, F. Mainardi (eds) Fractals and Fractional Calculus in Continuum Mechanics. International Centre for Mechanical Sciences (Courses and Lectures), vol. 378, Springer, Vienna, 1997, pp. 223--276. 
doi: https://doi.org/10.1007/978-3-7091-2664-6\_5

\bibitem{GOS18}
J. L. Gracia, E. O'Riordan, M. Stynes, 
Convergence in positive time for a finite difference method applied to a fractional convection-diffusion problem.
\emph{Comput. Meth. Appl. Math.} \textbf{18} (2018), 33--42.

\bibitem{HH98}
Y. Hatano, N. Hatano, 
Dispersive transport of ions in column experiments: An explanation of long-tailed profiles. 
\emph{Water Resour. Res.} \textbf{34} (1998), 1027--1033.

\bibitem{H81}
D. Henry, 
\emph{Geometric Theory of Semilinear Parabolic Equations}.
Springer-Verlag, Berlin Heidelberg, 1981.

\bibitem{Hilfer00}
R. Hilfer, 
Fractional time evolution. 
In:  E. Hilfer (ed) Applications of Fractional Calculus in Physics.  
World Science Publishing, River Edge, NJ, 2000, 87--130.

\bibitem{HLY19}
Z. Li, X. Huang, M. Yamamoto, 
Carleman estimates for the time-fractional advection-diffusion equations and applications. 
\emph{Inverse Problems} \textbf{35} (2019), 045003.

\bibitem{Ja01}
V. G. Jakubowski, 
Nonlinear elliptic-parabolic integro-differential equations with $L_1$-data: existence, uniqueness, asymptotics. 
Dissertation, University of Essen, Essen, Germany, 2001.

\bibitem{JLZ16}
B. Jin, R. Lazarov, Z. Zhou, 
Two schemes for fractional diffusion and diffusion wave equations with nonsmooth Data. 
\emph{SIAM J. Sci. Comput.} \textbf{38} (2016), A146--A170.

\bibitem{KSZ}
J. Kemppainen, J. Siljander, R. Zacher, 
Representation of solutions and large-time behavior for fully nonlocal diffusion equations. 
\emph{Journal of Differential Equations} \textbf{263} (2017), 149--201.

\bibitem{Ko08}
A. N. Kochubei, 
Distributed order calculus and equations of ultraslow diffusion. 
\emph{J. Math. Anal. Appl.} \textbf{340} (2008), 252--281.

\bibitem{Ko11}
A. N. Kochubei, 
General fractional calculus, evolution equations, and renewal processes. 
\emph{Integral Equations Operator Theory} \textbf{71} (2011), 583--600.

\bibitem{Kubica}
A. Kubica, K. Ryszewska, 
Decay of solutions to parabolic-type problem with distributed order Caputo derivative. 
\emph{Journal of Mathematical Analysis and Applications} \textbf{465} (2018), 75--99.

\bibitem{KY20}
A. Kubica, M. Yamamoto, 
Initial-boundary value problems for fractional diffusion equations with time-dependent coefficients. \emph{Fractional Calculus and Applied Analysis} \textbf{21} (2018), 276--311.

\bibitem{LHY20}
Z. Li, X. Huang, M. Yamamoto,  
Initial-boundary value problems for multi-term time-fractional diffusion equations with $x$-dependent coefficients. 
\emph{Evolution Equations and Control Theory} \textbf{9} (2020), 153--179.

\bibitem{LLY15}
Z. Li, Y. Liu, M. Yamamoto, 
Initial-boundary value problem for multi-term time-fractional diffusion equation with positive constants coefficients.
\emph{Appl. Math. Comput.} \textbf{257} (2015), 381--397.

\bibitem{LM1}
J.-L. Lions, E. Magenes, 
\emph{Non-homogeneous Boundary Value Problems and Applications}. 
Vol. 1, Springer-Verlag, Berlin, 1972.

\bibitem{L09}
Y. Luchko, 
Boundary value problems for the generalized time-fractional diffusion equation of distributed order. 
\emph{Fract. Calc. Appl. Anal.} \textbf{12} (2009), 409--422. 

\bibitem{Luchko09b}
Y. Luchko, 
Maximum principle for the generalized time-fractional diffusion equation. 
\emph{J. Math. Anal. Appl.} \textbf{351} (2009), 218--223. 


\bibitem{Luc11a}
Y. Luchko, A. Punzi,  
Modeling anomalous heat transport in geothermal reservoirs via fractional diffusion equations. 
\emph{International Journal on Geomathematics} \textbf{1} (2011), 257--276.

\bibitem{Luc11b}
Y. Luchko,
Anomalous diffusion models and their analysis. 
\emph{Forum der Berliner mathematischen Gesellschaft} \textbf{19} (2011), 53--85. 

\bibitem{LX16}
C. Lv, C. Xu, 
Error analysis of a high order method for time-fractional diffusion equations. 
\emph{SIAM J. Sci. Comput.} \textbf{38} (2016), A2699--A2724.

\bibitem{MK00}
R. Metzler, J. Klafter, 
The random walk's guide to anomalous diffusion: A fractional dynamics approach. 
\emph{Phys. Rep.} \textbf{339} (2000), 1--77.


\bibitem{P99}
I. Podlubny, 
\emph{Fractional Differential Equations}. 
Academic Press, San Diego, 1999.

\bibitem{P93}
J. Pruss, 
\emph{Evolutionary Integral Equations and Applications}. 
Monogr. Math. \textbf{87}, Birkhauser, Basel, 1993.

\bibitem{PS16}
H. K. Pang, H. W. Sun, 
Fast numerical contour integral method for fractional diffusion equations. 
\emph{J. Sci. Comput.} \textbf{66} (2016), 41--66.

\bibitem{RA94}
H. E. Roman, P. A. Alemany, 
Continuous-time random walks and the fractional diffusion equation. 
\emph{Journal of Physics A: Mathematical and General} \textbf{27} (1994), 3407.

\bibitem{SLC12}
S. Shen, F. Liu, J. Chen, et al., 
Numerical techniques for the variable order time fractional diffusion equation. 
\emph{Appl. Math. Comput.} \textbf{218} (2012), 10861--10870.

\bibitem{SY11}
K. Sakamoto, M. Yamamoto, 
Initial value/boundary value problems for fractional diffusion-wave equations and applications to 
some inverse problems. 
\emph{J. Math. Anal. Appl.} \textbf{382} (2011), 426--447.

\bibitem{SKB02}
I. Sokolov, J. Klafter, A. Blumen, 
Fractional kinetics. 
\emph{Phys. Today} \textbf{55} (2002), 48--54.

\bibitem{SZT17}
Z. Shi, Y. Zhao, Y. Tang, et al., 
Superconvergence analysis of an $H^1$-Galerkin mixed finite element method for two-dimensional multi-term time fractional diffusion equations. 
\emph{Int. J. of Comput. Math.} \textbf{95} (2018), 1845--1857.

\bibitem{T79}
R. Temam,
\emph{Navier-Stokes Equations}.
Revised edition, North-Holland, Netherlands, 1979.

\bibitem{U13}
V. V. Uchaikin, 
Fractional Derivatives for Physicists and Engineers I: Background and Theory. 
Nonlinear Physical Science, Springer, Heidelberg, 2013.
doi: https://doi.org/10.1007/978-3-642-33911-0

\bibitem{VZ}
V. Vergara, R. Zacher, 
Optimal decay estimates for time-fractional and other nonLocal subdiffusion equations via energy methods. 
\emph{SIAM Journal on Mathematical Analysis} \textbf{47} (2015), 210--239.

\end{thebibliography}
\end{document}